\pdfoutput=1
\RequirePackage{ifpdf}
\ifpdf 
\documentclass[pdftex]{sigma}
\else
\documentclass{sigma}
\fi

\numberwithin{equation}{section}

\newtheorem{thm}{Theorem}[section]
\newtheorem*{Theorem*}{Theorem}

\newtheorem{prop}[thm]{Proposition}
 { \theoremstyle{definition}
\newtheorem{defn}[thm]{Definition}

\newtheorem{rem}[thm]{Remark} }

\newcommand{\what}[1]{\widehat{#1}}

\newcommand{\ol}[1]{\overline{#1}}

\newcommand{\End}{\mathrm{End}}

\newcommand{\Hom}{\mathrm{Hom}}

\newcommand{\id}{\operatorname{id}}
\newcommand{\Id}{\operatorname{Id}}

\newcommand{\catMod}{\mathsf{Mod}}

\newcommand{\op}{\mathrm{op}}

\newcommand{\TCprod}{\overline{\boxtimes}}

\newcommand{\bC}{\mathbb{C}}

\newcommand{\bQ}{\mathbb{Q}}
\newcommand{\bR}{\mathbb{R}}

\newcommand{\bZ}{\mathbb{Z}}

\newcommand{\cA}{\mathcal{A}}
\newcommand{\cB}{\mathcal{B}}
\newcommand{\cC}{\mathcal{C}}

\newcommand{\cR}{\mathcal{R}}

\newcommand{\cT}{\mathcal{T}}
\newcommand{\cU}{\mathcal{U}}

\newcommand{\fsl}{\mathfrak{sl}}

\newcommand{\sfM}{\mathsf{M}}

\newcommand{\sfU}{\mathsf{U}}
\newcommand{\sfV}{\mathsf{V}}
\newcommand{\sfW}{\mathsf{W}}

\newcommand{\bbmi}{\mathbbm{i}}

\newcommand{\qint}[1]{\llbracket #1 \rrbracket}

\newcommand{\UEA}{\cU}

\newcommand{\Uq}{\UEA_{q}}
\newcommand{\Uqtilde}{\UEA_{\tilde{q}}}
\newcommand{\QGirrep}[1]{\sfM_{#1}}

\newcommand{\QGirvec}[2]{u^{(#1)}_{#2}}
\newcommand{\QGirmap}[1]{\pi_{#1}}

\newcommand{\QGcoprod}{\Delta}
\newcommand{\QGcounit}{\varepsilon}
\newcommand{\QGantipode}{S}
\newcommand{\CGemb}[3]{\iota_{#1}^{#2 #3}}
\newcommand{\CGproj}[3]{p^{#1}_{#2 #3}}
\newcommand{\CGcoeff}[4]{c_{#1}^{#2 #3}(#4)}
\newcommand{\CGembop}[3]{(\iota^{\op})_{#1}^{#2 #3}}
\newcommand{\sixjsymb}[6]{
\left\{
\begin{array}{ccc}
 #1 & #2 & #5 \\
 #3 & #4 & #6
\end{array}
\right\}
}

\newcommand{\ascmtx}[6]{\cB^{#1 #2 #5}_{#3 #4 #6}}

\newcommand{\catQG}{\cC_{\fsl_{2}}}
\newcommand{\catQGdouble}{\what{\cC}_{\fsl_{2}}}
\newcommand{\catVir}{\cC_{\vir}}
\newcommand{\catgen}{\cC}

\newcommand{\Sel}{\mathrm{Sel}}
\newcommand{\Chan}[4]{I^{#1 #2}_{#3 #4}}

\newcommand{\QGrepU}{\sfU}
\newcommand{\QGrepV}{\sfV}
\newcommand{\QGrepW}{\sfW}
\newcommand{\QGrepvec}{u}

\newcommand{\VecspU}{\mathbf{U}}
\newcommand{\VecspV}{\mathbf{V}}

\newcommand{\univR}{\cR}
\newcommand{\braid}{c}

\newcommand{\cattwist}{\theta}

\newcommand{\voaV}{V}

\newcommand{\voagen}{V}
\newcommand{\voamodgen}{W}
\newcommand{\voavac}{\mathbf{1}}
\newcommand{\voaY}{Y}

\newcommand{\voaconf}{\omega}
\newcommand{\voaint}{\mathcal{Y}}
\newcommand{\inttyp}[3]{\binom{#3}{#1 #2}}
\newcommand{\intsp}[3]{I\inttyp{#1}{#2}{#3}}
\newcommand{\intmap}{F}

\newcommand{\voavec}{v}

\newcommand{\vmodvec}{w}

\newcommand{\virvoa}{\mathbb{V}}
\newcommand{\virVerma}{\mathbb{M}}
\newcommand{\virmod}{\mathbb{W}}
\newcommand{\virmodgen}{\mathbb{U}}
\newcommand{\vir}{\mathfrak{vir}}
\newcommand{\hwvec}{\mathbf{w}}
\newcommand{\frweight}[1]{\mathbf{h}_{#1}}
\newcommand{\fusNorm}{\mathbf{B}}

\newcommand{\pzproj}[4]{(p_{P(#1)})_{#2 #3}^{#4}}

\newcommand{\ev}{\mathsf{ev}}
\newcommand{\coev}{\mathsf{coev}}

\usepackage{bm}
\usepackage{bbm}
\usepackage{eucal}
\usepackage{braket}
\usepackage{stmaryrd}
\usepackage{mathrsfs}
\usepackage{mathtools}

\usepackage[all,2cell]{xy}
\UseAllTwocells
\usepackage{youngtab}
\usepackage{ytableau}
\ytableausetup{smalltableaux}
\usepackage{tikz}
\usetikzlibrary{intersections,calc,arrows.meta}

\newcommand{\hackcenter}[1]{
 \xy (0,0)*{#1}; \endxy}

\begin{document}

\newcommand{\arXivNumber}{2207.12969}

\renewcommand{\PaperNumber}{039}

\FirstPageHeading

\ShortArticleName{Module Categories of the Generic Virasoro VOA and Quantum Groups}

\ArticleName{Module Categories of the Generic Virasoro VOA \\ and Quantum Groups}

\Author{Shinji KOSHIDA}

\AuthorNameForHeading{S.~Koshida}

\Address{Department of Mathematics and Systems Analysis, Aalto University, Finland}
\Email{\href{mailto:shinji.koshida@aalto.fi}{shinji.koshida@aalto.fi}}

\ArticleDates{Received October 25, 2024, in final form May 26, 2025; Published online June 03, 2025}

\Abstract{In this paper, we prove the equivalence between two ribbon tensor categories. On the one hand, we consider the category of modules of the Virasoro vertex operator algebra with generic central charge (generic Virasoro VOA) generated by those simple modules lying in the first row of the Kac table. On the other hand, we take the category of finite-dimensional type I modules of the quantum group $\mathcal{U}_q (\mathfrak{sl}_{2})$ with $q$ determined by the central charge. This is a continuation of our previous work in which we examined intertwining operators for the generic Virasoro VOA in detail. Our strategy to show the categorical equivalence is to take those results as input and directly compare the structures of tensor categories. Therefore, we are to execute the most elementary proof of categorical equivalence. We also study the category of $C_{1}$-cofinite modules of the generic Virasoro VOA. We show that it is ribbon equivalent to the category of finite-dimensional type I modules of $\mathcal{U}_q (\mathfrak{sl}_{2})\otimes \mathcal{U}_{\tilde{q}}(\mathfrak{sl}_{2})$, where $q$ and $\tilde{q}$ are again related to the central charge.}


\Classification{17B69; 17B68; 18M15}

\section{Introduction}
\subsection{Background}
Two-dimensional conformal field theories (CFT)~\cite{BPZ1984b,DMS_yellow_CFT} form one of the most developed classes of quantum field theories.
The infinite-dimensional symmetry of CFT allows one to identify the primary fields of the theory and the operator product expansion (OPE) of the primary fields in a closed form, sometimes leading to the exact solution.
Due to their integrability, CFTs have played important roles in string theory~\cite{GreenSchwarzWitten_superstring}, condensed matter physics~\cite{Hansson_QHall_CFT,Ludwig_CFT_condensed_matter}, and statistical physics~\cite{Cardy_lecture_2008,Mussardo-statistical_field_theory}.

Vertex operator algebras (VOA)~\cite{FLM-VOAs_Monster,Lepowsky_Li-VOA} and their representation theory give an algebraic counterpart of CFTs~\cite{Huang-CFT_and_VOA}.
According to the dictionary, a VOA gives the chiral algebra of a~CFT, and the irreducible modules of the VOA correspond to the primary fields of the CFT.
Furthermore, the OPE of primary fields is translated into a tensor product of modules of the VOA.
Therefore, we could say that the study of a CFT partly comes down to the study of the category of modules of the corresponding VOA as a tensor category.

Quantum groups~\cite{Drinfeld1986} are algebras whose module categories are clearly tensor categories.
There are several known examples of equivalence between module categories of VOAs and those of quantum groups as tensor categories.
Prominent examples include the Kazhdan--Lusztig-type duality~\cite{Drinfeld1989,Finkelberg1996,KL-tensor_structures_affine_Lie-I,KL-tensor_structures_affine_Lie-II,KL-tensor_structures_affine_Lie-III,KL-tensor_structures_affine_Lie-IV,McRae-nonneg_integer_level_affine_Lie_algebra_tensor_cat}, and the duality between the triplet VOA and the small quantum $\fsl_{2}$~\cite{CreutzigLentnerRupert,FeiginGainutdinovSemikhatovTipunin2006a,FeiginGainutdinovSemikhatovTipunin2006b,GannonNegron-quantum_sltwo_and_log_VOAs,KondoSaito2011, NagatomoTsuchiya2011,TsuchiyaWood2013}.
The present work is to give another example, namely, the Virasoro VOA with generic central charge and the quantum enveloping algebra of $\fsl_{2}$ with generic quantization parameter denoted by $\Uq(\fsl_{2})$.

\subsection{Relation to previous work}
This paper is a continuation of our previous work~\cite{KoshidaKytola2022}.
This work might have a different flavor compared to other works concerning categorical equivalence.
To be clear about our point of view and to motivate our method,
it would be worthwhile giving a quick overview of the line of research.

Since the early stages of the research of CFT, hidden quantum group symmetry has been observed in several places~\cite{FW-topological_representation_of_Uqsl2,GRS-quantum_groups_in_2d_physics,MR-comment_on_quantum_group_symmetry_in_CFT,PS-common_structures_between_finite_systems_and_CFTs,RRR-contour_picture_of_quantum_groups_in_CFT, SV-quantum_groups_and_homology_of_local_systems,
Varchenko-multidimensional_hypergeometric_functions_and_representation_theory_of_Lie_algebras_and_quantum_groups}.
In~\cite{KP-conformally_covariant_bdry_correlations}, the authors established one of the most concrete versions of the connection between a quantum group and CFT;
they gave maps from representations of $\Uq(\fsl_{2})$ to certain correlation functions of CFT.
Although their motivation was in application to the theory of Schramm--Loewner evolution~\cite{BB-CFTs_of_SLEs,Schramm2000}, their results were also of representation theoretical importance.
In fact, one consequence, among others, of their construction is that the asymptotic behavior of a correlation function when two points are close to each other is governed by the branching and associativity of the tensor product of representations of $\Uq(\fsl_{2})$.

In our previous work~\cite{KoshidaKytola2022}, we enhanced the results of \cite{KP-conformally_covariant_bdry_correlations} in the language of VOA.
Specifically, we identified the generic Virasoro VOA as the relevant VOA, determined the fusion rules among simple modules from the first row of the Kac table, and proved that the associativity of intertwining operators is governed by the quantum group $\Uq(\fsl_{2})$.

\subsection{Overview of the results}
In the present work, we examine the category-theoretical implications of the results of the previous one~\cite{KoshidaKytola2022},
and add discussion on the braiding as well.
It is standard to parametrize the central charge $c\in \bC$ of the universal Virasoro VOAs $\virvoa_{c}$ (see Section~\ref{sect:generic_VirVOA} for the definition)
by another parameter $t\in \bC$ as
$
	c = 13 - 6 \big(t + t^{-1}\big)$.
We focus on the case that $t\not\in\bQ$ is generic and call the VOA $\virvoa_{c}$ the generic Virasoro VOA.
The highest weight simple modules of $\virvoa_{c}$ are labeled by the conformal weight $h\in \bC$.
The Kac table gives the following {\it table} of conformal weights of particular interest
\begin{gather*}
	h_{r,s} = \frac{r^{2}-1}{4}t-\frac{rs-1}{2}+\frac{s^{2}-1}{4}t^{-1},\qquad r,s\in \bZ_{\geq 1}.
\end{gather*}
For each $h_{r,s}$, $r,s\in\bZ_{\geq 1}$, the corresponding simple highest weight $\virvoa_{c}$-module is denoted by $\virmod_{(r-1,s-1)}$.

In this paper, we follow~\cite{etingof2015tensor} for the categorical terminology.
In particular, a {\it tensor category} is a $\bC$-linear rigid monoidal category
with a one-dimensional endomorphism space on the unit object.
When it is further equipped with a braiding and a ribbon structure, we call it a {\it ribbon tensor category}.
Note that we do not assume finiteness.

We will study what we call the first-row category.
It is the category of $\virvoa_{c}$-modules generated by $\virmod_{(\ell,0)}$, $\ell \in\bZ_{\geq 0}$ as an additive category,
and will be denoted by $\catVir^{+}(t)$.
It is not a priori clear that the category $\catVir^{+}(t)$ can be equipped with the structure of a ribbon tensor category.
Thus, we first show that it is a ribbon tensor category.
Furthermore, we relate it to another ribbon tensor category of interest;
the category of finite-dimensional $\Uq(\fsl_{2})$-modules of type I denoted by $\catQG (q)$.
Let us phrase our result in the following way.

\begin{thm}
\label{thm:first-row-cat_QG}
The first-row category $\catVir^{+}(t)$ is equipped with the structure of a ribbon tensor category.
The resulting ribbon tensor category $\catVir^{+}(t)$ is equivalent to $\catQG(q)$
under the parameter matching $q = {\rm e}^{\pi \bbmi t}$.
\end{thm}

We also study the category $\catVir^{1}$ of $C_{1}$-cofinite $\virvoa_{c}$-modules.
It has been proven in~\cite{CJORY-tensor_categories_arising_from_Virasoro_algebra} that the category $\catVir^{1}$ is semi-simple and the simple objects are exhausted by $\virmod_{(k,l)}$, $k,l \in\bZ_{\geq 0}$.
Furthermore, it is equipped with the structure of a ribbon tensor category,
and the resulting ribbon tensor category will be denoted by $\catVir^{1}(t)$.
In particular, the first-row category $\catVir^{+}(t)$ is a subcategory of $\catVir^{1}(t)$.
It also follows from the fusion rules in $\catVir^{1}(t)$ that, as a tensor category,
the category $\catVir^{1}(t)$ is generated by the modules in $\catVir^{+}(t)$ and the modules in the first {\it column} of the Kac table,
which allows for the following result.

\begin{thm}
\label{thm:C1-cat_QG}
As a ribbon tensor category, $\catVir^{1}(t)$ is equivalent to the category of finite-dimensional type I modules of $\Uq (\fsl_{2})\otimes \cU_{\tilde{q}}(\fsl_{2})$ with the parameter matching $q = {\rm e}^{\pi\bbmi t}$, \smash{$\tilde{q} = {\rm e}^{\pi \bbmi t^{-1}}$}.
\end{thm}

\subsection{Relation to other works}
There are, in fact, various ways to prove our results.
First of all, in \cite{CJORY-tensor_categories_arising_from_Virasoro_algebra},
they proved that $\catVir^{+}(t)$ is braided equivalent to a certain Kazhdan--Lusztig category of \smash{$\what{\fsl}_{2}$} along the way of proving the rigidity of $\catVir^{+}(t)$.
Therefore, our results can be derived from it, too.

There are also approaches to the Kazhdan--Lusztig-type duality when the quantum group is at a root of unity.
In \cite{GannonNegron-quantum_sltwo_and_log_VOAs}, it has been established that the module category of the quantum~${\rm SL}(2)$ at a root of unity and a certain module category of the Virasoro VOA in the logarithmic setting are equivalent.
The key tool therein is Ostrik's functor~\cite{ostrik2005module} that characterizes the module category of the quantum~${\rm SL}(2)$.
We also mention~\cite{McRaeYang2022} that studied the $\fsl_{2}$-type structure in the module category of the Virasoro VOA at central charge $25$.
The recent work~\cite{creutzig2023algebraic,lentner2025conditional} is building a general, and perhaps conceptual, framework towards the Kazhdan--Lusztig duality.
These works address the case when the quantum group parameter is a root of unity, but the methods would extend to the generic case as well.

In spite that the above mentioned works would cover our results, we believe that it is still worthwhile recording our proof because our method is different;
the feature of our method is that it relies on an actual construction of intertwining operators.
Another recent example of a~similar approach (to different problems) can be found in~\cite{nakano2024fusion}.

{\bf Organization of the paper.}
In the following Section~\ref{sect:quantum_group}, we give an account of the quantum group $\Uq(\fsl_{2})$
and fix several details of, e.g., the Clebsch--Gordan coefficients and the universal $R$-matrix.
In Section~\ref{sect:module_cat_VOA}, we recall the necessary background information on the module category of a VOA.
In particular, we quickly look at the Huang--Lepowsky theory that equips the module category of a VOA with the structure of a ribbon tensor category.
We focus our attention on the generic Virasoro VOA in Section~\ref{sect:generic_VirVOA} and review the known results from our previous work~\cite{KoshidaKytola2022}.
Section~\ref{sect:first-row_module_category} is the main part of this paper, where we define the first row module category of the generic Virasoro VOA and examine its structure in detail. Consequently, we will prove Theorem~\ref{thm:first-row-cat_QG}.
In Section~\ref{sect:C1-cofinite_module_category}, we study the category of $C_{1}$-cofinite modules of the generic Virasoro VOA and prove Theorem~\ref{thm:C1-cat_QG}.

\section{Quantum group}
\label{sect:quantum_group}
In this section, we give a brief overview of known facts about the quantum group $\Uq (\fsl_{2})$ and its representations.
We only consider the case where $q\in\bC^{\times}$ is not a root of unity and simply write $\Uq$ for $\Uq (\fsl_{2})$.
Further details about quantum groups can be found in \cite{chari1995guide, Kassel-quantum_groups,Lusztig1993}.

\subsection[Quantum group U\_q]{Quantum group $\boldsymbol{\Uq}$}
The algebra $\Uq$ is a unital associative $\bC$-algebra generated by
$K$, $K^{-1}$, $E$, and $F$ subject to the relations
\[
 KK^{-1}=K^{-1}K=1, \qquad
 KE=q^{2}EK, \qquad KF=q^{-2}FK, \qquad
 EF-FE=\frac{K-K^{-1}}{q-q^{-1}}.
 \]
We equip~$\Uq$ with the structure of a Hopf algebra in the following way.
The coproduct $\QGcoprod\colon \Uq\to \Uq\otimes\Uq$ is given by
\[
 \QGcoprod (K) :=K\otimes K, \qquad
 \QGcoprod (E) :=E\otimes 1+K\otimes E, \qquad
 \QGcoprod (F) :=F\otimes K^{-1} +1\otimes F.
 \]
The counit $\QGcounit \colon \Uq\to\bC$ and the antipode $\QGantipode\colon \Uq\to\Uq$ are defined as
\begin{gather*}
 \QGcounit (E)=\QGcounit (F)=0, \qquad \QGcounit (K)=1, \qquad
 \QGantipode (E)=-K^{-1}E, \\ \QGantipode (F)=-FK, \qquad \QGantipode (K)=K^{-1}.
\end{gather*}

\begin{rem}
The coproduct used in our previous work~\cite{KoshidaKytola2022} was the opposite: $\QGcoprod^{\op}=P\circ \QGcoprod$, where
$
 P\colon \Uq\otimes\Uq \to \Uq\otimes \Uq$; $ A\otimes B\to B\otimes A$
is the permutation of tensor components.
As we shall see, the structure constants of associativity ($6j$-symbols; see below) manifestly observed are associated with the opposite coproduct $\QGcoprod^{\op}$, but we will come back to the original $\QGcoprod$ in the end.
\end{rem}

\subsection{Irreducible representations}
For each $\ell\in \bZ_{\geq 0}$, $\QGirrep{\ell}$ is an $(\ell+1)$-dimensional complex vector space with a basis \smash{$\big(\QGirvec{\ell}{i}\big)_{i=0,1,\dots, \ell}$}.
We define a representation homomorphism
$
 \QGirmap{\ell}\colon \Uq\to \End (\QGirrep{\ell})
$
by
\begin{gather*}
 \QGirmap{\ell} (K) \QGirvec{\ell}{i} =q^{\ell-2i}\QGirvec{\ell}{i}, \qquad
 \QGirmap{\ell}(E) \QGirvec{\ell}{i} = \qint{i}\qint{\ell-i+1} \QGirvec{\ell}{i-1},\qquad
 \QGirmap{\ell} (F) \QGirvec{\ell}{i} = \QGirvec{\ell}{i+1}
\end{gather*}
for $i=0,1,\dots,\ell$.
Here, we understand \smash{$\QGirvec{\ell}{i}=0$} when $i<0$ or $i>\ell$.
The $q$-integers are defined~by
\[
 \qint{n}:=\frac{q^{n}-q^{-n}}{q-q^{-1}}, \qquad n\in \bZ.
\]

It is known that $(\QGirmap{\ell},\QGirrep{\ell})$, $\ell\in \bZ_{\geq 0}$ are irreducible representations of $\Uq$.
On the other hand, a~finite-dimensional irreducible representation of $\Uq$ is isomorphic to either $(\QGirmap{\ell},\QGirrep{\ell})$ or~${(\QGirmap{\ell}\circ\chi,\QGirrep{\ell})}$, where $\chi$ is the automorphism of $\Uq$ given by $\chi (K)=-K$, $\chi (E)=-E$, $\chi (F)=F$.
In the sequel, we always assume that the vector space $\QGirrep{\ell}$ is equipped with the representation homomorphism~$\QGirmap{\ell}$ and will not specify it. These irreducible representations $\QGirrep{\ell}$, $\ell\in\bZ_{\geq 0}$ are often referred to as type I.
Otherwise, the type I representations are characterized by that the eigenvalues of $K$ are of the form $q^{n}$ with some $n\in\bZ$.

Clearly, $\QGirrep{\ell}$ is generated by \smash{$\QGirvec{\ell}{0}$}, which we call a highest weight vector of highest weight $\ell$.
Accordingly, we say that $\QGirrep{\ell}$ is a highest weight irreducible representation of highest weight $\ell$.

\subsection{Tensor product}
For $\ell_{1}$, $\ell_{2}$, the tensor product $\QGirrep{\ell_{1}}\otimes \QGirrep{\ell_{2}}$ is equipped with the structure of a representation of~$\Uq$ by the coproduct $\QGcoprod$.
To emphasize the dependence on the coproduct, we write $\QGirrep{\ell_{1}}\otimes_{\QGcoprod}\QGirrep{\ell_{2}}$ for the tensor product representation.
It decomposes into irreducible representations according to the Clebsch--Gordan rule
\begin{gather}
\label{eq:QG_CGrule}
 \QGirrep{\ell_{1}}\otimes_{\QGcoprod} \QGirrep{\ell_{2}}\simeq \bigoplus_{\ell\in\Sel (\ell_{1},\ell_{2})}\QGirrep{\ell},
\end{gather}
where we defined the set
\begin{gather*}
 \Sel (\ell_{1},\ell_{2})= \{\ell\in \bZ_{\geq 0}\mid |\ell_{1}-\ell_{2}|\leq \ell \leq \ell_{1}+\ell_{2},\, \ell+\ell_{1}+\ell_{2}\equiv 0 \bmod 2 \}
\end{gather*}
of the highest weights appearing in the tensor product.

For each triple $(\ell,\ell_{1},\ell_{2})$ such that $\ell_{1},\ell_{2}\in \bZ_{\geq 0}$ and $\ell\in \Sel (\ell_{1},\ell_{2})$, we fix an embedding homomorphism \smash{$\CGemb{\ell}{\ell_{1}}{\ell_{2}}\colon \QGirrep{\ell}\to \QGirrep{\ell_{1}}\otimes_{\QGcoprod} \QGirrep{\ell_{2}}$} so that
\begin{gather}
 \CGemb{\ell}{\ell_{1}}{\ell_{2}}\big(\QGirvec{\ell}{0}\big)=\sum_{j=0}^{s}\CGcoeff{\ell}{\ell_{1}}{\ell_{2}}{j} \QGirvec{\ell_{1}}{j}\otimes \QGirvec{\ell_{2}}{s-j}, \nonumber \\
 \CGcoeff{\ell}{\ell_{1}}{\ell_{2}}{j}=(-1)^{j}\frac{\qint{\ell_{1}-j}!}{\qint{j}!\qint{s-j}!}\frac{\qint{\ell_{2}-s+j}!}{\qint{\ell_{1}}!\qint{\ell_{2}}!}\frac{q^{j(\ell_{1}-j+1)}}{\big(q-q^{-1}\big)^{s}},\qquad j=0,1,\dots, s,\label{eq:CG_coeffs}
\end{gather}
where we set $s=(\ell_{1}+\ell_{2}-\ell)/2$ and the $q$-factorials are defined by
\begin{gather*}
 \qint{n}!:=
 \begin{cases}
 \qint{n}\qint{n-1}\cdots\qint{1},& n\in\bZ_{> 0},\\
 1, &n=0.
 \end{cases}
\end{gather*}
Accordingly, the family of projections $\CGproj{\ell}{\ell_{1}}{\ell_{2}}\colon \QGirrep{\ell_{1}}\otimes_{\QGcoprod} \QGirrep{\ell_{2}}\to \QGirrep{\ell}$, $\ell_{1},\ell_{2}\in\bZ_{\geq 0}$, $\ell\in\Sel (\ell_{1},\ell_{2})$ is determined by the properties that
\begin{gather*}
 \CGproj{\ell}{\ell_{1}}{\ell_{2}}\circ \CGemb{\ell'}{\ell_{1}}{\ell_{2}}=
 \begin{cases}
 \id_{\QGirrep{\ell}}, & \ell=\ell', \\
 0, & \ell\neq\ell',
 \end{cases}\qquad \ell,\ell'\in \Sel (\ell_{1},\ell_{2}), \\
 \sum_{\ell\in\Sel(\ell_{1},\ell_{2})}\CGemb{\ell}{\ell_{1}}{\ell_{2}}\circ \CGproj{\ell}{\ell_{1}}{\ell_{2}}=\id_{\QGirrep{\ell_{1}}\otimes\QGirrep{\ell_{2}}},
\end{gather*}
for each $\ell_{1},\ell_{2}\in\bZ_{\geq 0}$.

Recall that there is another coproduct $\QGcoprod^{\op}$ on $\Uq$, with which we can form another tensor product representation $\QGirrep{\ell_{1}}\otimes_{\QGcoprod^{\op}}\QGirrep{\ell_{2}}$ for $\ell_{1},\ell_{2}\in\bZ_{\geq 0}$.
The rule of decomposition into irreducible representations is the same as in \eqref{eq:QG_CGrule}.
For each $\ell\in \Sel(\ell_{1},\ell_{2})$, we obtain an injective homomorphism $\CGembop{\ell}{\ell_{1}}{\ell_{2}}\colon \QGirrep{\ell}\to \QGirrep{\ell_{1}}\otimes_{\QGcoprod^{\op}}\QGirrep{\ell_{2}}$ by
\smash{$
 \CGembop{\ell}{\ell_{1}}{\ell_{2}}:=P_{\QGirrep{\ell_{2}},\QGirrep{\ell_{1}}}\circ \CGemb{\ell}{\ell_{2}}{\ell_{1}}$}.
Here, we write $P_{\VecspU,\VecspV}$ with vector spaces $\VecspU$ and $\VecspV$ for the permutation operator
$
 P_{\VecspU,\VecspV}\colon \VecspU\otimes\VecspV\to \VecspV\otimes\VecspU$; $ u\otimes v\mapsto v\otimes u$.

\subsection{Representation category}
Here we describe the category of $\Uq$-modules as a ribbon tensor category.
The general theory of tensor categories can be found in \cite{BakalovKirillovJr2001,etingof2015tensor}.

\subsubsection{Linear category}
We write $\catMod (\Uq)$ for the category of finite-dimensional representations of $\Uq$, which is clearly a~$\bC$-linear category.
As we have already noted, the simple objects of $\catMod(\Uq)$ are exhausted by~$\QGirrep{\ell}$,~$\ell\in\bZ_{\geq 0}$ and their twist by the automorphism $\chi$ up to isomorphism.
Under our assumption that $q$ is not a root of unity, it is also known that $\catMod(\Uq)$ is semi-simple although it has infinitely many simple objects.

We define $\catQG$ as the full subcategory of $\catMod(\Uq)$ generated by $\QGirrep{\ell}$, $\ell\in\bZ_{\geq 0}$.
In other words, each object of $\catQG$ is isomorphic to a finite direct sum of $\QGirrep{\ell}$, $\ell\in\bZ_{\geq 0}$.

\subsubsection{Tensor structure}
We can equip the category $\catMod (\Uq)$ with the monoidal bifunctor $-\otimes_{\QGcoprod}-\colon \catMod(\Uq)\times\catMod(\Uq)\to \catMod(\Uq)$ defined by means of the coproduct $\QGcoprod$.
The associativity isomorphisms are given by
\begin{gather*}
 \alpha_{\QGrepU,\QGrepV,\QGrepW}\colon\ (\QGrepU\otimes_{\QGcoprod}\QGrepV)\otimes_{\QGcoprod} \QGrepW\to \QGrepU\otimes_{\QGcoprod} (\QGrepV\otimes_{\QGcoprod}\QGrepW);\qquad (u\otimes v)\otimes w\mapsto u\otimes (v\otimes w)
\end{gather*}
for $\QGrepU,\QGrepV,\QGrepW\in\catMod (\Uq)$,
i.e., the associativity for the underlying vector spaces.
We can take $\QGirrep{0}$ as a unit object with respect to this tensor product and choose unit isomorphisms $\lambda_{\QGrepU}\colon \QGirrep{0}\otimes_{\QGcoprod} \QGrepU\to\QGrepU$ and $\rho_{\QGrepU}\colon \QGrepU\otimes_{\QGcoprod}\QGirrep{0}\to\QGrepU$.
In this way, $(\otimes_{\QGcoprod},\alpha,\QGirrep{0},\lambda,\rho)$ defines a monoidal structure on the category~$\catMod (\Uq)$.
Since $\Uq$ has an invertible antipode, $\catMod(\Uq)$ is also rigid.
Therefore, $\catMod(\Uq)$ is a~tensor category.

Since the subcategory $\catQG$ is closed under the tensor product and the dual, it is a tensor subcategory of $\catMod (\Uq)$.
We suppose that the unit isomorphisms are chosen in such a way that~${\lambda_{\QGirrep{\ell}}=\CGproj{\ell}{0}{\ell}}$ and $\rho_{\QGirrep{\ell}}=\CGproj{\ell}{\ell}{0}$ for all $\ell\in\bZ_{\geq 0}$.
Then, the monoidal structure on $\catQG$ is uniquely determined.
As a monoidal category, $\catQG$ starts depending on the parameter $q$.
Thus, we write~$\catQG (q)$ for the category $\catQG$ equipped with the above tensor structure.

If we take the other coproduct $\QGcoprod^{\op}$, the same underlying category $\catQG$ is equipped with another monoidal structure $(\otimes_{\QGcoprod^{\op}},\alpha^{\op},\QGirrep{0},\lambda^{\op},\rho^{\op})$.
This tensor category will be denoted by~$\catQG^{\op}(q)$.
Note that, however, as a linear map, each associativity isomorphism $\alpha^{\op}_{\QGrepU,\QGrepV,\QGrepW}$, $\QGrepU,\QGrepV,\QGrepW\in \catQG$, can be taken as the same $\alpha_{\QGrepU,\QGrepV,\QGrepW}$.

At this point, we can discuss the so-called $6j$-symbols.
In our context, the $6j$-symbols of~$\catQG^{\op}(q)$, instead of those of $\catQG (q)$, will be more manifest.
For an arbitrary choice of four $\ell_{1},\ell_{2},\ell_{3},\ell_{4}\in\bZ_{\geq 0}$, we define the set
\smash{$
 \Chan{\ell_{1}}{\ell_{2}}{\ell_{3}}{\ell_{4}}:=\Sel (\ell_{1},\ell_{2})\cap\Sel (\ell_{3},\ell_{4})$}.
Now, we fix $\ell_{1},\ell_{2},\ell_{3},\ell_{4}\in\bZ_{\geq 0}$
and compare two spaces of homomorphisms:
one is $\Hom_{\Uq}(\QGirrep{\ell_{4}},(\QGirrep{\ell_{1}}\otimes_{\QGcoprod^{\op}} \QGirrep{\ell_{2}})\otimes_{\QGcoprod^{\op}} \QGirrep{\ell_{3}})$, which has a basis
\begin{gather*}
 \big(\CGembop{m}{\ell_{1}}{\ell_{2}}\otimes_{\QGcoprod^{\op}} \id_{\QGirrep{\ell_{3}}}\big)\circ \CGembop{\ell_{4}}{m}{\ell_{3}}
 =\hackcenter{
 \begin{tikzpicture}
 \draw[thick] (0,-1)--(0,0);
 \draw[thick] (0,0)--(-1,1);
 \draw[thick] (0,0)--(1,1);
 \draw[thick] (0,1)--(-.5,.5);
 \draw (0,-1)node[below]{$\ell_{4}$};
 \draw (-1,1)node[above]{$\ell_{1}$};
 \draw (0,1)node[above]{$\ell_{2}$};
 \draw (1,1)node[above]{$\ell_{3}$};
 \draw (-.25,.25)node[below left]{$m$};
 \end{tikzpicture}
 },\qquad m\in \Chan{\ell_{1}}{\ell_{2}}{\ell_{3}}{\ell_{4}}
\end{gather*}
and the other is $\Hom_{\Uq}(\QGirrep{\ell_{4}},\QGirrep{\ell_{1}}\otimes_{\QGcoprod^{\op}} (\QGirrep{\ell_{2}}\otimes_{\QGcoprod^{\op}} \QGirrep{\ell_{3}}))$, which has a basis
\begin{gather*}
 \big(\id_{\QGirrep{\ell_{1}}}\otimes_{\QGcoprod^{\op}} \CGembop{n}{\ell_{2}}{\ell_{3}}\big)\circ \CGembop{\ell_{4}}{\ell_{1}}{n}
 =\hackcenter{
 \begin{tikzpicture}
 \draw[thick] (0,-1)--(0,0);
 \draw[thick] (0,0)--(-1,1);
 \draw[thick] (0,0)--(1,1);
 \draw[thick] (0,1)--(.5,.5);
 \draw (0,-1)node[below]{$\ell_{4}$};
 \draw (-1,1)node[above]{$\ell_{1}$};
 \draw (0,1)node[above]{$\ell_{2}$};
 \draw (1,1)node[above]{$\ell_{3}$};
 \draw (.25,.25)node[below right]{$n$};
 \end{tikzpicture}
 },\qquad n\in \Chan{\ell_{2}}{\ell_{3}}{\ell_{1}}{\ell_{4}}.
\end{gather*}
Here we also drew diagrams representing the homomorphisms.
Each trivalent vertex depicts an injection and the composition is read from bottom to top.
The associativity isomorphism induces an isomorphism of these vector spaces of morphisms
\begin{gather*}
 \Hom_{\Uq}(\QGirrep{\ell_{4}},(\QGirrep{\ell_{1}}\otimes_{\QGcoprod^{\op}} \QGirrep{\ell_{2}})\otimes_{\QGcoprod^{\op}} \QGirrep{\ell_{3}})\to \Hom_{\Uq}(\QGirrep{\ell_{4}},\QGirrep{\ell_{1}}\otimes_{\QGcoprod^{\op}} (\QGirrep{\ell_{2}}\otimes_{\QGcoprod^{\op}} \QGirrep{\ell_{3}})), \\
 f \mapsto \alpha^{\op}_{\QGirrep{\ell_{1}},\QGirrep{\ell_{2}},\QGirrep{\ell_{3}}}\circ f.
\end{gather*}
Our $6j$-symbols
\[
\sixjsymb{\ell_{1}}{\ell_{2}}{\ell_{3}}{\ell_{4}}{m}{n},\qquad m\in \Chan{\ell_{1}}{\ell_{2}}{\ell_{3}}{\ell_{4}},\qquad n\in \Chan{\ell_{2}}{\ell_{3}}{\ell_{1}}{\ell_{4}}
\]
 are defined as the matrix elements of the {\it inverse} of this isomorphism in terms of the bases introduced above.
Explicitly, they are defined by
\begin{gather}
 \big(\alpha^{\op}_{\QGirrep{\ell_{1}},\QGirrep{\ell_{2}},\QGirrep{\ell_{3}}}\big)^{-1}\circ \big(\id_{\QGirrep{\ell_{1}}}\otimes_{\QGcoprod^{\op}} \CGembop{n}{\ell_{2}}{\ell_{3}}\big)\circ \CGembop{\ell_{4}}{\ell_{1}}{n}\nonumber \\
 \qquad =
 \sum_{m\in \Chan{\ell_{1}}{\ell_{2}}{\ell_{3}}{\ell_{4}}}
 \sixjsymb{\ell_{1}}{\ell_{2}}{\ell_{3}}{\ell_{4}}{m}{n}
 \big(\CGembop{m}{\ell_{1}}{\ell_{2}}\otimes_{\QGcoprod^{\op}} \id_{\QGirrep{\ell_{3}}}\big)\circ \CGembop{\ell_{4}}{m}{\ell_{3}}\label{eq:QG_6j_definition}
\end{gather}
for \smash{$n\in \Chan{\ell_{2}}{\ell_{3}}{\ell_{1}}{\ell_{4}}$}, or diagramatically
\begin{gather*}
 \big(\alpha^{\op}_{\QGirrep{\ell_{1}},\QGirrep{\ell_{2}},\QGirrep{\ell_{3}}}\big)^{-1}\circ -\colon\
 \hackcenter{
 \begin{tikzpicture}
 \draw[thick] (0,-1)--(0,0);
 \draw[thick] (0,0)--(-1,1);
 \draw[thick] (0,0)--(1,1);
 \draw[thick] (0,1)--(.5,.5);
 \draw (0,-1)node[below]{$\ell_{4}$};
 \draw (-1,1)node[above]{$\ell_{1}$};
 \draw (0,1)node[above]{$\ell_{2}$};
 \draw (1,1)node[above]{$\ell_{3}$};
 \draw (.25,.25)node[below right]{$n$};
 \end{tikzpicture}
 }\mapsto
 \sum_{m\in \Chan{\ell_{1}}{\ell_{2}}{\ell_{3}}{\ell_{4}}}
 \sixjsymb{\ell_{1}}{\ell_{2}}{\ell_{3}}{\ell_{4}}{m}{n}
 \hackcenter{
 \begin{tikzpicture}
 \draw[thick] (0,-1)--(0,0);
 \draw[thick] (0,0)--(-1,1);
 \draw[thick] (0,0)--(1,1);
 \draw[thick] (0,1)--(-.5,.5);
 \draw (0,-1)node[below]{$\ell_{4}$};
 \draw (-1,1)node[above]{$\ell_{1}$};
 \draw (0,1)node[above]{$\ell_{2}$};
 \draw (1,1)node[above]{$\ell_{3}$};
 \draw (-.25,.25)node[below left]{$m$};
 \end{tikzpicture}
 } .
\end{gather*}

\begin{rem}
The dependence of $\catQG^{\op}(q)$ (and $\catQG(q)$ as well) on the parameter $q$ is manifest from the fact that the $6j$-symbols depend on $q$.
There are numeric presentations of the $6j$-symbols found in~\cite{frenkel1997canonical,masbaum19943},
though they are not necessary for our purpose.
\end{rem}

\subsubsection{Braiding and twist}
Here we consider the quantum group $\Uq$ to be equipped with the coproduct $\QGcoprod$ (but not $\QGcoprod^{\op}$.) It admits the universal $R$-matrix given by the following formula:
\begin{gather*}
 \cR=q^{\frac{1}{2}H\otimes H}\sum_{n=0}^{\infty}\frac{q^{n(n-1)/2}\big(q-q^{-1}\big)^{n}}{\qint{n}!}(F^{n}\otimes E^{n}),
\end{gather*}
which makes sense in a certain completion of $\Uq\otimes\Uq$.
Here, $H$ is the symbol that is supposed to behave as $K=q^{H}$.
Its action on $\QGrepU\otimes_{\QGcoprod}\QGrepV$ with $\QGrepU,\QGrepV\in \catQG(q)$ is well defined.
Indeed, since these representations are finite-dimensional, the infinite sum in the formula of $\univR$ truncates to a~finite sum, and for $u\in\QGrepU$ and $v\in \QGrepV$ such that $K.u=q^{m}u$ and $K.v=q^{n}v$, we may understand
\begin{gather*}
 q^{\frac{1}{2}H\otimes H}.u\otimes v=q^{\frac{mn}{2}}u\otimes v.
\end{gather*}
We write $R_{\QGrepU,\QGrepV}$ for the action of $\univR$ on $\QGrepU\otimes_{\QGcoprod}\QGrepV$
and set
\begin{gather*}
 \braid_{\QGrepU,\QGrepV}:=P_{\QGrepU,\QGrepV}R_{\QGrepU,\QGrepV}\colon\ \QGrepU\otimes_{\QGcoprod}\QGrepV\to\QGrepV\otimes_{\QGcoprod}\QGrepU.
\end{gather*}
The family of morphisms $\braid_{\QGrepU,\QGrepV}$, $\QGrepU,\QGrepV\in \catQG (q)$ gives a braiding structure to $\catQG (q)$.

There is a standard procedure to read a twist isomorphism out of the universal $R$-matrix~\cite[Chapter~XIV.6]{Kassel-quantum_groups}.
For each $\QGrepU\in\catQG (q)$, we can define the twist isomorphism $\cattwist_{\QGrepU}\colon \QGrepU\to \QGrepU$ by
\begin{gather}
\label{eq:QGtwist_inverse}
 \cattwist_{\QGrepU}^{-1}\QGrepvec=(-1)^{H}K\sum_{n=0}^{\infty}\frac{q^{n(n-1)/2}\big(q-q^{-1}\big)^{n}}{\qint{n}!}\big(-K^{-1}E\big)^{n}q^{-\frac{1}{2}H^{2}}F^{n}\QGrepvec,\qquad \QGrepvec\in\QGrepU.
\end{gather}
It turns out that this twist isomorphism gives a ribbon structure.

\section{Category of modules for a VOA}
\label{sect:module_cat_VOA}
In this section, we review the construction of a tensor product, braiding and a ribbon structure on the category of modules for a VOA developed in \cite{Huang1995,Huang2005,HuangLepowsky1992,HuangLepowsky1994,Huang--LepowskyI,Huang--LepowskyII,Huang--LepowskyIII,huang1999intertwining}.
For a concise account, \cite[Section 2]{Huang--Kirillov--Lepowsky} is also helpful.

\subsection{VOA and modules}
Here, we record the definitions of a VOA and its modules following, e.g.,~\cite{FLM-VOAs_Monster,Lepowsky_Li-VOA}, but equivalent sets of axioms are also found in~\cite{FrenkelBen-Zvi2004,Kac-vertex_algebras}.

First, we recall the Virasoro algebra $\vir$.
It is the infinite-dimensional Lie algebra $\vir=\bigoplus_{n\in\bZ}\bC L_{n}\oplus \bC C$ with the Lie bracket
\begin{align*}
 [L_{m},L_{n}]=(m-n)L_{m+n}+\frac{m^{3}-m}{12}\delta_{m+n,0}C,\qquad m,n\in\bZ, \qquad
 [C,\vir] = 0.
\end{align*}

We also need to fix a convention with formal calculus to phrase the definition of a VOA.
That is, the formal series $(x-y)^{n}$ with $n\in\mathbb{Z}$ in the two variables $x$ and $y$ is defined by
\begin{align*}
	(x-y)^{n} \coloneqq \sum_{k=0}^{\infty}\binom{n}{k}x^{n-k}(-y)^{k}.
\end{align*}
In particular, $(x-y)^{n}$ and $(-1)^{n}(y-x)^{n}$ are different formal series when $n<0$.
We need this convention to interpret the expression such as
$
	\delta \big(\frac{x_{1}-x_{2}}{x_{0}}\big)$.
Here, $\delta (x) \coloneqq \sum_{n\in\bZ}x^{n}$ is called the formal $\delta$-distribution,
and we understand
\begin{align*}
	\delta \left(\frac{x_{1}-x_{2}}{x_{0}}\right)
	=\sum_{n\in\bZ}x_{0}^{-n}(x_{1}-x_{2})^{n}
	=\sum_{n\in\bZ}\sum_{k=0}^{\infty}\binom{n}{k}x_{0}^{-n}x_{1}^{n-k}(-x_{2})^{k}.
\end{align*}

The definition of a VOA goes as follows.
\begin{defn}
A VOA consists of a $\bZ$-graded vector space $\voagen=\bigoplus_{n\in\bZ}\voagen_{n}$ such that $\dim\voagen_{n}<\infty$, $n\in\bZ$,
distinguished vectors $\voavac\in\voagen_{0}$ and $\voaconf\in\voagen_{2}$,
and a linear map
\begin{align*}
	\voaY(-,x)\colon \ \voagen\to \End(\voagen)\big[\big[x^{\pm 1}\big]\big],\qquad
	 a \mapsto \voaY (a,x) = \sum_{n\in \bZ}a_{(n)}x^{-n-1}
\end{align*}
subject to the following axioms:
\begin{enumerate}\itemsep=0pt
\item[(1)] 	(Field condition) For all $a,b\in \voaV$, $a_{(n)}b=0$ for all sufficiently large $n\gg 0$.
\item[(2)] 	(Vacuum axiom) $\voaY(\voavac,x) = \id$ and
$
		\voaY (a,x)\voavac \in a + \voaV [[x]]x
$
	for all $a\in \voaV$.
\item[(3)] 	(Virasoro representation) Define \smash{$L^{\voaV}_{n} = \voaconf_{(n+1)}\in \End(\voaV)$}, $n\in\bZ$. Then, the assignment $L_{n}\mapsto L^{\voaV}_{n}$, $n\in\bZ$, $C\mapsto c\cdot \id$ with a scalar $c\in \bC$
gives a representation of $\vir$.
Furthermore, $L^{\voaV}_{0}|_{\voaV_{n}}=n\id_{\voaV_{n}}$, $n\in\bZ$ and $L^{\voaV}_{-1}a = a_{(-2)}\voavac$, $a\in \voaV$.
\item[(4)] 	(Translation property) For all $a\in \voaV$,
	\begin{align*}
		\big[L_{-1}^{\voaV}, \voaY(a,x)\big] = \frac{\rm d}{{\rm d}x}\voaY(a,x).
	\end{align*}
\item[(5)] 	(Jacobi identity) For all $a,b\in \voaV$,
	\begin{gather*}
 x_{0}^{-1}\delta \left(\frac{x_{1}-x_{2}}{x_{0}}\right)\voaY(a,x_{1})\voaY (b,x_{2})-x_{0}^{-1}\delta\left(\frac{x_{2}-x_{1}}{-x_{0}}\right)\voaY (b,x_{2})\voaY(a,x_{1}) \\
 \qquad = x_{2}^{-1}\delta\left(\frac{x_{1}-x_{0}}{x_{2}}\right)\voaY\left(\voaY(a,x_{0})b,x_{2}\right).
\end{gather*}
\end{enumerate}
\end{defn}

The distinguished vectors $\voavac$ and $\voaconf$ are called the vacuum and the conformal vectors, respectively.
The linear map $\voaY(-,x)$ is called the state-field correspondence map, and
the scalar $c$ that appeared in the Virasoro representation is called the central charge of the VOA.
Though a VOA is a quadruple $(\voaV, \voavac,\voaconf,\voaY)$,
it is common to only refer to $\voaV$ to address a VOA.
We will introduce the particular example of our interest in Section~\ref{sect:generic_VirVOA}.

The modules of a VOA are defined as follows.
\begin{defn}
For a VOA $\voagen$, a $\voagen$-module is a pair of a $\bC$-graded vector space $\voamodgen=\bigoplus_{\alpha \in\bC}\voamodgen_{\alpha}$ such that $\dim\voamodgen_{\alpha}<\infty$ and $W_{\alpha-n}=0$, $n\gg 0$ for all $\alpha\in\bC$ and a linear map
\begin{gather*}
	\voaY_{\voamodgen}(-,x)\colon \ \voagen \to \End (\voamodgen) \big[\big[x^{\pm 1}\big]\big], \qquad
	 a \mapsto \voaY_{\voamodgen}(a,x) = \sum_{n\in\bZ}a^{\voamodgen}_{(n)}x^{-n-1}
\end{gather*}
subject to the following axioms:
\begin{enumerate}\itemsep=0pt
\item[(1)] 	(Field condition) For all $a\in \voagen$ and $w\in \voamodgen$, $a^{\voamodgen}_{(n)}w=0$ for all sufficiently large $n\gg 0$.
\item[(2)] 	(Vacuum axiom) $\voaY_{\voamodgen}(\voavac,x)=\id_{\voamodgen}$.
\item[(3)] 	($L_{0}$-action) For all $\alpha\in\bC$, $\voamodgen_{\alpha}$ is a generalized eigenspace of \smash{$L_{0}^{\voamodgen}\coloneqq\voaconf^{\voamodgen}_{(1)}$} with eigenvalue~$\alpha$. In other words, $L^{\voamodgen}_{0}-\alpha$ is nilpotent on $\voamodgen_{\alpha}$.
\item[(4)] 	(Jacobi identity) For all $a,b\in \voaV$,
	\begin{gather*}
 x_{0}^{-1}\delta \left(\frac{x_{1}-x_{2}}{x_{0}}\right)\voaY_{\voamodgen}(a,x_{1})\voaY_{\voamodgen} (b,x_{2})-x_{0}^{-1}\delta\left(\frac{x_{2}-x_{1}}{-x_{0}}\right)\voaY_{\voamodgen} (b,x_{2})\voaY_{\voamodgen}(a,x_{1}) \\
\qquad = x_{2}^{-1}\delta\left(\frac{x_{1}-x_{0}}{x_{2}}\right)\voaY_{\voamodgen}\left(\voaY(a,x_{0})b,x_{2}\right).
\end{gather*}
\end{enumerate}
\end{defn}

When we set \smash{$L^{\voamodgen}_{n}=\voaconf^{\voamodgen}_{(n+1)}$}, $n\in\bZ$,
and $c$ to be the same central charge as $\voagen$ itself,
the assignment $L_{n}\mapsto L_{n}^{\voamodgen}$, $n\in\bZ$, $C\mapsto c\cdot \id_{\voamodgen}$ gives a representation of $\vir$.

\subsection{Intertwining operators/maps}
Finally, we come to the notion of intertwining operators.
\begin{defn}
Let $\voagen$ be a VOA and $\voamodgen_{1}$, $\voamodgen_{2}$, $\voamodgen_{3}$ be $\voagen$-modules.
An intertwining operator of type \smash{$\inttyp{\voamodgen_{1}}{\voamodgen_{2}}{\voamodgen_{3}}$} is a linear map
\begin{gather*}
 \voaint (-,x)\colon \ \voamodgen_{1}\to\Hom (\voamodgen_{2},\voamodgen_{3})\{x\}
\end{gather*}
satisfying the Jacobi identity
\begin{gather*}
 x_{0}^{-1}\delta \left(\frac{x_{1}-x_{2}}{x_{0}}\right)\voaY_{\voamodgen_{3}}(\voavec,x_{1})\voaint (\vmodvec_{1},x_{2})-x_{0}^{-1}\delta\left(\frac{x_{2}-x_{1}}{-x_{0}}\right)\voaint (\vmodvec_{1},x_{2})\voaY_{\voamodgen_{2}}(\voavec,x_{1}) \\
 \qquad = x_{2}^{-1}\delta\left(\frac{x_{1}-x_{0}}{x_{2}}\right)\voaint\left(\voaY_{\voamodgen_{1}}(\voavec,x_{0})\vmodvec_{1},x_{2}\right)
\end{gather*}
for all $\voavec\in \voagen$ and $\vmodvec_{1}\in\voamodgen_{1}$,
and the $L_{-1}$-derivation property
\begin{gather*}
 \voaint\big(L_{-1}^{\voamodgen_{1}}\vmodvec_{1},x\big)=\frac{\rm d}{{\rm d}x}\voaint (\vmodvec_{1},x)
\end{gather*}
for all $\vmodvec_{1}\in\voamodgen_{1}$.
The set of intertwining operators of type \smash{$\inttyp{\voamodgen_{1}}{\voamodgen_{2}}{\voamodgen_{3}}$} is a vector space denoted by~\smash{$\intsp{\voamodgen_{1}}{\voamodgen_{2}}{\voamodgen_{3}}$}, whose dimension is called the fusion rule of that type.
\end{defn}

\begin{rem}
\label{rem:intertw_special_cases}
If all three modules are $\voagen$ itself, $\intsp{\voagen}{\voagen}{\voagen}$ contains the state-field correspondence map $\voaY(-,x)$.
If $\voamodgen_{1}=\voagen$ and $\voamodgen_{2}=\voamodgen_{3}=\voamodgen$, then the space of intertwining operators $\intsp{\voagen}{\voamodgen}{\voamodgen}$ contains the module map $\voaY_{\voamodgen}(-,x)$.
\end{rem}

Later, we will need the following transformation of intertwining operators.
Let $\voaint$ be an intertwining operator of type $\inttyp{\voamodgen_{1}}{\voamodgen_{2}}{\voamodgen_{3}}$ with $\voamodgen_{1}$, $\voamodgen_{2}$, and $\voamodgen_{3}$ being modules for a VOA $\voagen$.
Define
\begin{gather}
 (\Omega\voaint)(-,x)\colon\ \voamodgen_{2}\to\Hom (\voamodgen_{1},\voamodgen_{3})\{x\}\label{eq:int_flip}
\end{gather}
by
\begin{gather*}
 (\Omega\voaint)(\vmodvec_{2},x)\vmodvec_{1}:={\rm e}^{xL^{\voamodgen_{3}}_{-1}}\voaint\big(\vmodvec_{1},{\rm e}^{\pi\bbmi}x\big)\vmodvec_{2},\qquad \vmodvec_{1}\in\voamodgen_{1}, \qquad \vmodvec_{2}\in\voamodgen_{2}.
\end{gather*}
Then, we can show~\cite[Proposition~7.1]{Huang--LepowskyII} that $\Omega\voaint$ is an intertwining operator of type $\inttyp{\voamodgen_{2}}{\voamodgen_{1}}{\voamodgen_{3}}$.
Note that $\Omega$ is invertible.

For $z\in\bC^{\times}$, we define $\log (z)$ so that $\arg (z)\in [0,2\pi)$.
Then specializing the formal variable in an intertwining operator $\voaint$ at $x={\rm e}^{\log (z)}$ makes sense, giving rise to the linear map
\begin{gather*}
 \intmap_{\voaint,z}\colon\ \voamodgen_{1}\otimes\voamodgen_{2}\to\ol{\voamodgen_{3}};\qquad \vmodvec_{1}\otimes\vmodvec_{2}\to \voaint (\vmodvec_{1},x)\vmodvec_{2}|_{x={\rm e}^{\log (z)}},
\end{gather*}
where we define the completion of $\voamodgen_{3}$ by \smash{$\ol{\voamodgen_{3}}=\prod_{n\in\bZ_{\geq 0}}(\voamodgen_{3})_{n}$}.
In the sequel, we simply write
\begin{gather*}
 \voaint (\vmodvec_{1},z)\vmodvec_{2}=\voaint (\vmodvec_{1},x)\vmodvec_{2}|_{x={\rm e}^{\log (z)}}
\end{gather*}
for the evaluation of the formal variable as long as the branch is chosen as $\arg z\in [0,2\pi)$.
We call such a linear map obtained from an intertwining operator an {\it intertwining map} associated to~$z$.
Since the correspondence $\voaint\mapsto \intmap_{\voaint,z}$ is one-to-one, we could say that they are the same notion, but intertwining maps are still more convenient when we define the tensor product of modules.

\subsection[P(z)-tensor product]{$\boldsymbol{ P(z)}$-tensor product}
Let $\voamodgen_{1}$ and $\voamodgen_{2}$ be $\voagen$-modules, and fix $z\in\bC^{\times}$.
A $P(z)$-product of $\voamodgen_{1}$ and $\voamodgen_{2}$ is a $\voagen$-module~$\voamodgen_{3}$ together with an intertwining map $\intmap_{z}$ of type $\inttyp{\voamodgen_{1}}{\voamodgen_{2}}{\voamodgen_{3}}$ associated to $z$.
The $P(z)$-{\it tensor} product is a universal object among $P(z)$-products.
To be precise, the $P(z)$-tensor product of~$\voamodgen_{1}$ and $\voamodgen_{2}$ is the $P(z)$-product $(\voamodgen_{1}\boxtimes_{P(z)}\voamodgen_{2}, \boxtimes_{P(z)})$
such that for any $P(z)$-product $(\voamodgen_{3},\intmap_{z})$ of $\voamodgen_{1}$ and $\voamodgen_{2}$, there exists a unique $\voagen$-module homomorphism $\eta\colon \voamodgen_{1}\boxtimes_{P(z)}\voamodgen_{2}\to \voamodgen_{3}$ satisfying~${\intmap_{z}=\ol{\eta}\circ \boxtimes_{P(z)}}$
\begin{align*}
\xymatrix{
\voamodgen_{1}\otimes\voamodgen_{2} \ar[r]^{\boxtimes_{P(z)}} \ar[dr]_{F_{z}} & \voamodgen_{1}\boxtimes_{P(z)}\voamodgen_{2} \ar@{..>}[d]^{\exists!\eta} \\
& \voamodgen_{3}.
}
\end{align*}
It is standard to show that the $P(z)$-tensor product is unique up to isomorphism if it exists.
We note that the $P(z)$-tensor product depends on the choice of a category of modules.

For the $P(z)$-tensor product $(\voamodgen_{1}\boxtimes_{P(z)}\voamodgen_{2}, \boxtimes_{P(z)})$,
there exists a unique intertwining operator~${\voaint (-,x)}$ of type~\smash{$\inttyp{\voamodgen_{1}}{\voamodgen_{2}}{\voamodgen_{1}\boxtimes_{P(z)}\voamodgen_{2}}$} such that
$
 \vmodvec_{1}\boxtimes_{P(z)}\vmodvec_{2}=\voaint (\vmodvec_{1},z)\vmodvec_{2}$, $ \vmodvec_{1}\in\voamodgen_{1}, \vmodvec_{2}\in\voamodgen_{2}$.

The $P(z)$-tensor product defines a bifunctor on the category of modules of interest.
\!To~see~this, let $\voamodgen_{1}$, $\voamodgen_{2}$, $\voamodgen_{3}$, $\voamodgen_{4}$ be $\voagen$-modules and let $f\colon \voamodgen_{1}\to\voamodgen_{3}$ and $g\colon\voamodgen_{2}\to\voamodgen_{4}$ be morphisms.
Suppose that the $P(z)$-tensor products \smash{$\big(\voamodgen_{1}\boxtimes_{P(z)}\voamodgen_{2},\boxtimes_{P(z)}^{12}\big)$} and \smash{$\big(\voamodgen_{3}\boxtimes_{P(z)}\voamodgen_{4},\boxtimes_{P(z)}^{34}\big)$} exist.
Then the object $\voamodgen_{3}\boxtimes_{P(z)}\voamodgen_{4}$ together with $\boxtimes_{P(z)}^{34}\circ (f\otimes g)$ gives a $P(z)$-product of $\voamodgen_{1}$ and~$\voamodgen_{2}$.
This means that there exists a unique morphism denoted by $f\boxtimes_{P(z)} g\colon \voamodgen_{1}\boxtimes_{P(z)}\voamodgen_{2}\to \voamodgen_{3}\boxtimes_{P(z)}\voamodgen_{4}$ such that
\begin{gather*}
 \boxtimes^{34}_{P(z)}\circ (f\otimes g)=\ol{f\boxtimes_{P(z)}g}\circ \boxtimes_{P(z)}^{12}.
\end{gather*}

\subsection{Composition and iteration}
Let $\voamodgen_{1}$, $\voamodgen_{2}$, $\voamodgen_{3}$ be $\voagen$-modules and $z_{1},z_{2}\in \bC^{\times}$ be such that $|z_{1}|>|z_{2}|>|z_{1}-z_{2}|>0$.
In this setting, let us assume that the composition
$
 \voamodgen_{1}\boxtimes_{P(z_{1})}(\voamodgen_{2}\boxtimes_{P(z_{2})}\voamodgen_{3})
$
exists.
Then, associated with it, we have the composition of intertwining maps
\begin{gather*}
 \voaint^{1}(-,z_{1})\voaint^{2}(-,z_{2})-\colon\ \voamodgen_{1}\otimes \voamodgen_{2}\otimes\voamodgen_{3}\to \ol{\voamodgen_{1}\boxtimes_{P(z_{1})}(\voamodgen_{2}\boxtimes_{P(z_{2})}\voamodgen_{3})},
\end{gather*}
where
\begin{gather*}\voaint^{1}(-,x)\in \intsp{\voamodgen_{1}}{\voamodgen_{2}\boxtimes_{P(z_{2})}\voamodgen_{3}}{\voamodgen_{1}\boxtimes_{P(z_{1})}(\voamodgen_{2}\boxtimes_{P(z_{2})}\voamodgen_{3})}\qquad \mbox{and}\qquad
\voaint^{2}(-,x)\in \intsp{\voamodgen_{2}}{\voamodgen_{3}}{\voamodgen_{2}\boxtimes_{P(z_{2})}\voamodgen_{3}}
\end{gather*}
are the corresponding intertwining operators.
For $\vmodvec_{i}\in\voamodgen_{i}$, $i=1,2,3$, we can identify
\begin{gather*}
 \vmodvec_{1}\boxtimes_{P(z_{1})}(\vmodvec_{2}\boxtimes_{P(z_{2})}\vmodvec_{3})=\voaint^{1}(\vmodvec_{1},z_{1})\voaint^{2}(\vmodvec_{2},z_{2})\vmodvec_{3}.
\end{gather*}
Note that composition of intertwining maps does not automatically make sense,
but we must verify that the infinite sum appearing in the composition absolutely converges.

Let us also assume that the iteration
$
 (\voamodgen_{1}\boxtimes_{P(z_{1}-z_{2})}\voamodgen_{2})\boxtimes_{P(z_{2})}\voamodgen_{3}
$
exists.
In this case, we get the iteration of the corresponding intertwining maps
\begin{gather*}
 \voaint_{1} (\voaint_{2}(-,z_{1}-z_{2})-,z_{2} )-\colon\ \voamodgen_{1}\otimes \voamodgen_{2}\otimes \voamodgen_{3}\to \ol{(\voamodgen_{1}\boxtimes_{P(z_{1}-z_{2})}\voamodgen_{2})\boxtimes_{P(z_{2})}\voamodgen_{3}},
\end{gather*}
where
\begin{gather*}
 \voaint_{1}(-,x)\in\intsp{\voamodgen_{1}\boxtimes_{P(z_{1}-z_{2})}\voamodgen_{2}}{\voamodgen_{3}}{(\voamodgen_{1}\boxtimes_{P(z_{1}-z_{2})}\voamodgen_{2})\boxtimes_{P(z_{2})}\voamodgen_{3}},\qquad
 \voaint_{2}(-,x)\in \intsp{\voamodgen_{1}}{\voamodgen_{2}}{\voamodgen_{1}\boxtimes_{P(z_{1}-z_{2})}\voamodgen_{2}}.
\end{gather*}
We can make the identification
\begin{gather*}
 (\vmodvec_{1}\boxtimes_{P(z_{1}-z_{2})}\vmodvec_{2})\boxtimes_{P(z_{2})}\vmodvec_{3}=\voaint_{1}\left(\voaint_{2}(\vmodvec_{1},z_{1}-z_{2})\vmodvec_{2},z_{2}\right)\vmodvec_{3}
\end{gather*}
for $\vmodvec_{i}\in \voamodgen_{i}$, $i=1,2,3$.
Similarly to the composition, we need to verify that the iteration of intertwining maps makes sense.

The associativity isomorphism
\begin{gather*}
 \big(\cA_{P(z_{1}),P(z_{2})}^{P(z_{1}-z_{2}),P(z_{2})}\big)_{\voamodgen_{1},\voamodgen_{2},\voamodgen_{3}}\colon\
 (\voamodgen_{1}\boxtimes_{P(z_{1}-z_{2})}\voamodgen_{2})\boxtimes_{P(z_{2})}\voamodgen_{3}\to
 \voamodgen_{1}\boxtimes_{P(z_{1})}(\voamodgen_{2}\boxtimes_{P(z_{2})}\voamodgen_{3}),
\end{gather*}
if it exists, is characterized by the property
\begin{gather*}
 \ol{\big(\cA_{P(z_{1}),P(z_{2})}^{P(z_{1}-z_{2}),P(z_{2})}\big)_{\voamodgen_{1},\voamodgen_{2},\voamodgen_{3}}}\colon\
 (\vmodvec_{1}\boxtimes_{P(z_{1}-z_{2})}\vmodvec_{2})\boxtimes_{P(z_{2})}\vmodvec_{3} \mapsto
 \vmodvec_{1}\boxtimes_{P(z_{1})}(\vmodvec_{2}\boxtimes_{P(z_{2})}\vmodvec_{3})
\end{gather*}
for $\vmodvec_{i}\in \voamodgen_{i}$, $i=1,2,3$.

\subsection{Parallel transport}
Let $\voamodgen_{1}$, $\voamodgen_{2}$ be $\voagen$-modules,
and let us take $z_{1},z_{2}\in\bC^{\times}$.
Then we may consider the $P(z_{1})$-tensor product $\voamodgen_{1}\boxtimes_{P(z_{1})}\voamodgen_{2}$ and the $P(z_{2})$-tensor product $\voamodgen_{1}\boxtimes_{P(z_{2})}\voamodgen_{2}$.
Depending on a path $\gamma$ in $\bC^{\times}$ from $z_{1}$ to $z_{2}$,
the parallel transport isomorphism $\cT_{\gamma}\colon \voamodgen_{1}\boxtimes_{P(z_{1})}\voamodgen_{2}\to \voamodgen_{1}\boxtimes_{P(z_{2})}\voamodgen_{2}$ is defined as follows.

As before, we fix a branch of $\log (z_{2})$ so that $\arg (z_{2})\in [0,2\pi)$.
Then we write $l_{\gamma}(z_{1})$ for the logarithm of $z_{1}$ determined by the analytic continuation along $\gamma$ from $\log (z_{2})$.
Let $\voaint (-,x)$ be the corresponding intertwining operator to the $P(z_{2})$-tensor product $\voamodgen_{1}\boxtimes_{P(z_{2})}\voamodgen_{2}$.
Then, the parallel transport $\cT_{\gamma}$ is characterized by the property
\begin{gather*}
 \ol{\cT}_{\gamma}\left(\vmodvec_{1}\boxtimes_{P(z_{1})}\vmodvec_{2}\right)=\voaint(\vmodvec_{1},x)\vmodvec_{2}|_{x={\rm e}^{l_{\gamma}(z_{1})}},\qquad \vmodvec_{1}\in\voamodgen_{1}, \vmodvec_{2}\in\voamodgen_{2}.
\end{gather*}

\subsection{Monoidal structure}
The $P(z)$-tensor product depends on $z\in \bC^{\times}$, so varying $z$, we get a family of tensor products that are related by parallel transport.
Here, we fix a single monoidal structure at $z=1$.

First, we set $\boxtimes=\boxtimes_{P(1)}$ and take $\voagen$ as a unit object.
For a $\voagen$-module $\voamodgen$, the unit isomorphisms $\lambda_{\voamodgen}\colon \voagen\boxtimes\voamodgen\to\voamodgen$ and $\rho_{\voamodgen}\colon\voamodgen\boxtimes \voagen\to \voamodgen$ are characterized by
\begin{gather*}
 \lambda_{\voamodgen}\colon\ \voavac\boxtimes \vmodvec\mapsto \vmodvec,\qquad \ol{\rho_{\voamodgen}}\colon\ \vmodvec\boxtimes\voavac\mapsto {\rm e}^{L^{\voamodgen}_{-1}}\vmodvec,\qquad \vmodvec\in\voamodgen.
\end{gather*}

To define the associativity isomorphism, we take $z_{1}$ and $z_{2}$ on the real axis so that $z_{1}>z_{2}>z_{1}-z_{2}>0$.
We also take several paths in $\bR_{>0}$: $\gamma_{1}$ from $1$ to $z_{2}$, $\gamma_{2}$ from $1$ to $z_{1}-z_{2}$, $\gamma_{3}$ from~$z_{1}$ to $1$, and $\gamma_{4}$ from $z_{2}$ to $1$.
Then, given three $\voagen$-modules $\voamodgen_{i}$, $i=1,2,3$, the isomorphism $\cA_{\voamodgen_{1},\voamodgen_{2},\voamodgen_{3}}\colon (\voamodgen_{1}\boxtimes \voamodgen_{2})\boxtimes \voamodgen_{3}\to \voamodgen_{1}\boxtimes (\voamodgen_{2}\boxtimes \voamodgen_{3})$ is the following compositions of isomorphisms:
\begin{gather*}
 (\voamodgen_{1}\boxtimes \voamodgen_{2})\boxtimes \voamodgen_{3}
 \xrightarrow{\cT_{\gamma_{1}}}
 (\voamodgen_{1}\boxtimes \voamodgen_{2})\boxtimes_{P(z_{2})} \voamodgen_{3}
 \xrightarrow{\cT_{\gamma_{2}}\boxtimes_{P(z_{2})}\id_{\voamodgen_{3}}} \\
 (\voamodgen_{1}\boxtimes_{P(z_{1}-z_{2})} \voamodgen_{2})\boxtimes_{P(z_{2})} \voamodgen_{3}
 \xrightarrow{(\cA_{P(z_{1}),P(z_{2})}^{P(z_{1}-z_{2}),P(z_{2})})_{\voamodgen_{1},\voamodgen_{2},\voamodgen_{3}}}
 \voamodgen_{1}\boxtimes_{P(z_{1})} (\voamodgen_{2}\boxtimes_{P(z_{2})} \voamodgen_{3}) \\
 \xrightarrow{\id_{\voamodgen_{1}}\boxtimes_{P(z_{1})}\cT_{\gamma_{4}}}
 \voamodgen_{1}\boxtimes_{P(z_{1})} (\voamodgen_{2}\boxtimes \voamodgen_{3})
 \xrightarrow{\cT_{\gamma_{3}}}
 \voamodgen_{1}\boxtimes (\voamodgen_{2}\boxtimes \voamodgen_{3}).
\end{gather*}
Then, $(\boxtimes, \cA, \voagen, \lambda,\rho)$ gives a monoidal structure of the category of interest.

\subsection{Braiding and twist}
\label{sect:HL_braiding}
We can define a braiding on the category by means of the parallel transport.
Let $\gamma$ be a path from~$-1$ to $1$ contained in the complex upper half plane except its end points.
Given two $\voagen$-modules $\voamodgen_{1}$ and $\voamodgen_{2}$, the braiding isomorphism $c_{\voamodgen_{1},\voamodgen_{2}}\colon \voamodgen_{1}\boxtimes\voamodgen_{2}\to \voamodgen_{2}\boxtimes \voamodgen_{1}$ is characterized~by
\begin{gather*}
 \ol{c_{\voamodgen_{1},\voamodgen_{2}}}\colon\ \vmodvec_{1}\boxtimes\vmodvec_{2}\mapsto {\rm e}^{L^{\voamodgen_{2}\boxtimes\voamodgen_{1}}_{-1}}\ol{\cT_{\gamma}}(\vmodvec_{2}\boxtimes_{P(-1)}\vmodvec_{1}),\qquad \vmodvec_{1}\in\voamodgen_{1}, \vmodvec_{2}\in \voamodgen_{2}.
\end{gather*}

For each $\voagen$-module $\voamodgen$, the twist $\theta_{\voamodgen}\colon \voamodgen\to\voamodgen$ defined by \smash{$\theta_{\voamodgen}={\rm e}^{2\pi\bbmi L_{0}^{\voamodgen}}$} gives a ribbon structure.

\section{Generic Virasoro VOA and modules}
\label{sect:generic_VirVOA}
In this section, we introduce the generic Virasoro VOA and its modules, and review the results of our previous work~\cite{KoshidaKytola2022} on the first-row modules.

\subsection{Generic Virasoro VOA and the Kac table}
For a fixed central charge $c\in \bC$, the universal Virasoro VOA $\virvoa_{c}$ is given by
\begin{gather*}
 \virvoa_{c}=\UEA (\vir)/\big(\UEA (\vir)(C-c)+\textstyle{\sum_{n\geq -1}}\UEA (\vir)L_{n}\big)
\end{gather*}
together with the vacuum vector $\voavac=[1]$ and the conformal vector $\voaconf=[L_{-2}]$.
The state field correspondence map $\voaY(-,x)$ is uniquely determined by
\begin{gather*}
 \voaY(\voaconf,x)=\sum_{n\in\bZ}L_{n}x^{-n-2}.
\end{gather*}

We parameterize the central charge by another parameter $t$ as
\begin{gather}
\label{eq:vir_cc_in_t}
 c=c(t)=13-6\big(t+t^{-1}\big).
\end{gather}
It is known that, when $t\not\in\bQ$, $\virvoa_{c}$ is a simple VOA (see, e.g.,~\cite{IK-representation_theory_of_the_Virasoro_algebra}).
In this case, we call $\virvoa_{c}$ the generic Virasoro VOA of central charge $c$.

Let us consider modules of $\virvoa_{c}$.
For a conformal weight $h\in\bC$, the Verma module $\virVerma (c,h)$ is given by
\begin{gather*}
 \virVerma (c,h)= \UEA (\vir)/\big(\UEA (\vir)(C-c)+\UEA(\vir)(L_{0}-h)+\textstyle{\sum_{n\geq 1}}\UEA(\vir)L_{n}\big).
\end{gather*}
The Verma module is not only a representation of $\vir$, but also a $\virvoa_{c}$-module.
Under the parametrization \eqref{eq:vir_cc_in_t}, the conformal weights of the Kac table are given by
\begin{gather*}
	h_{r,s} = h_{r,s}(t)= \frac{r^{2}-1}{4}t-\frac{rs-1}{2}+\frac{s^{2}-1}{4}t^{-1},\qquad r,s\in \bZ_{\geq 1}.
\end{gather*}
The Verma module $\virVerma (c(t),h_{k+1,\ell +1}(t))$ with $k,\ell\in\bZ_{\geq 0}$ is reducible, so we write its simple quotient as $\virmod_{(k,\ell)}$.
It is known that $\virmod_{(k,\ell)}$ are self-dual: $\virmod_{(k,\ell)}^{\vee}\simeq \virmod_{(k,\ell)}$, $k,\ell\in\bZ_{\geq 0}$.
We also set $\virmod_{\ell}:= \virmod_{(\ell,0)}$, $\ell\in\bZ_{\geq 0}$ and call them the first-row modules.
We remark that $\virmod_{0}$ is $\virvoa_{c}$ itself because the maximal proper submodule of $\virVerma (c(t),h_{1,1}(t))$ is generated by $[L_{-1}]$.

In the rest of this section, we focus on the first-row modules.
It will then be convenient to write
\begin{gather}
\label{eq:conf_weight_fr}
	\frweight{\ell}:= h_{\ell+1,1} = \frac{\ell(\ell+2)}{4}t-\frac{\ell}{2},\qquad \ell\in\bZ_{\geq 0}
\end{gather}
for the conformal weights of the first-row modules.
We also fix the highest weight vector $\hwvec_{\ell}$ of~$\virmod_{\ell}$ as the image of $1\in\UEA(\vir)$.

\subsection{Fusion rules}
The fusion rules among the first-row modules have been known for a long time~\cite{FrenkelZhu2012}, and we gave an alternative proof for them in our previous work~\cite{KoshidaKytola2022}.
Let us record the result here.

\begin{thm}[\cite{FrenkelZhu2012,KoshidaKytola2022}]
\label{thm:first-row_module_fusion}
For $\ell_{1},\ell_{2},\ell_{3}\in \bZ_{\geq 0}$,
\begin{gather*}
 \dim \intsp{\virmod_{\ell_{1}}}{\virmod_{\ell_{2}}}{\virmod_{\ell_{3}}}=
 \begin{cases}
 1, & \ell_{3}\in \Sel (\ell_{1},\ell_{2}),\\
 0, & \mathrm{otherwise}.
 \end{cases}
\end{gather*}
Recall that $\Sel$ is the selection rule set of the Clebsch--Gordan rule.
\end{thm}

Suppose that $\ell_{1},\ell_{2},\ell_{3}\in\bZ_{\geq 0}$ satisfy the selection rule $\ell_{3}\in\Sel (\ell_{1},\ell_{2})$.
Then there exists a~unique intertwining operator of type \smash{$\inttyp{\virmod_{\ell_{1}}}{\virmod_{\ell_{2}}}{\virmod_{\ell_{3}}}$} up to constant.
We fix the normalization of the intertwining operator
\[
\voaint_{\ell_{1} \ell_{2}}^{\ell_{3}}(-,x)\in \intsp{\virmod_{\ell_{1}}}{\virmod_{\ell_{2}}}{\virmod_{\ell_{3}}}
\]
 as
\begin{gather*}
 \voaint_{\ell_{1} \ell_{2}}^{\ell_{3}}(\hwvec_{\ell_{1}},x)\hwvec_{\ell_{2}}\in \fusNorm_{\ell_{1} \ell_{2}}^{\ell_{3}}\hwvec_{\ell_{3}}x^{\frweight{\ell_{3}}-\frweight{\ell_{1}}-\frweight{\ell_{2}}}+\virmod_{\ell_{3}}[[x]]x^{\frweight{\ell_{3}}-\frweight{\ell_{1}}-\frweight{\ell_{2}}+1},
\end{gather*}
where the constant $\fusNorm_{\ell_{1} \ell_{2}}^{\ell_{3}}$ is given by the formula
\begin{gather*}
 \fusNorm_{\ell_{1} \ell_{2}}^{\ell_{3}}=\frac{1}{s!}\prod_{j=1}^{s}\frac{\Gamma (1+tj) \Gamma (1-t(\ell_{1}+1-j)) \Gamma (1-t(\ell_{2}+1-j))}{\Gamma (1+t) \Gamma (2-t(2-p+\ell_{1}+\ell_{2}-s))},
\end{gather*}
where $s=(\ell_{1}+\ell_{2}-\ell_{3})/2$.

\begin{rem}
\label{rem:intertw_Virasoro_special_cases}
Let us remark on a few properties of the intertwining operator \smash{$\voaint_{\ell_{1} \ell_{2}}^{\ell_{3}}(-,x)$}~when $\ell_{1}=0$ or $\ell_{2}=0$.
When $\ell_{1}=0$, as we pointed out in Remark~\ref{rem:intertw_special_cases}, \smash{$\intsp{\virvoa_{c}}{\virmod_{\ell}}{\virmod_{\ell}}$} contains~$\voaY_{\virmod_{\ell}}(-,x)$, which now must span the space of intertwining operators.
Furthermore, we can identify $\voaint_{0 \ell}^{\ell}(-,x)\allowbreak=\voaY_{\virmod_{\ell}}(-,x)$ by observing
$
 \voaint_{0 \ell}^{\ell}(\voavac,x) =\Id_{\virmod_{\ell}}=\voaY_{\virmod_{\ell}}(\voavac,x)$.
For the case where $\ell_{2}=0$, the normalization of $\voaint_{\ell 0}^{\ell}(-,x)$ gives us
$
 \voaint_{\ell 0}^{\ell}(\hwvec_{\ell},x)\voavac\in\hwvec_{\ell}+\virmod_{\ell}[[x]]x$.
By the Jacobi identity and the $L_{-1}$-derivation property, we can deduce that
\smash{$
 \voaint_{\ell 0}^{\ell}(\hwvec_{\ell},x)\voavac={\rm e}^{L_{-1}^{\virmod_{\ell}}}\hwvec_{\ell}$}.
\end{rem}

\subsection{Associativity of intertwining operators}
The general idea of associativity comes down to comparing the composition and iteration of intertwining operators,
but it is not even clear in general if the composition and iteration are possible.
The following theorem is a consequence of the general analysis by Huang~\cite{Huang2005},
or was proven in our previous work~\cite{KoshidaKytola2022}.

\begin{thm}[\cite{Huang2005,KoshidaKytola2022}]
Let $\ell_{1},\ell_{2},\ell_{3},\ell_{4}\in\bZ_{\geq 0}$.
\begin{enumerate}\itemsep=0pt
 \item[$(1)$] For any $n\in \Chan{\ell_{2}}{\ell_{3}}{\ell_{1}}{\ell_{4}}$ and $\vmodvec_{1}\in \virmod_{\ell_{1}}$, $\vmodvec_{2}\in\virmod_{\ell_{2}}$, $\vmodvec_{3}\in\virmod_{\ell_{3}}$, the formal series
 \begin{gather*}
 \voaint_{\ell_{1} n}^{\ell_{4}}(\vmodvec_{1},x_{1})\voaint_{\ell_{2} \ell_{3}}^{n}(\vmodvec_{2},x_{2})\vmodvec_{3}
 \end{gather*}
 in $x_{1}$ and $x_{2}$ converges in $\ol{\virmod}_{\ell_{4}}$ at \smash{$x_{1}={\rm e}^{\log (z_{1})}$}, \smash{$x_{2}={\rm e}^{\log (z_{2})}$} such that $|z_{1}|>|z_{2}|>0$.
 \item[$(2$] For any $m\in \Chan{\ell_{1}}{\ell_{2}}{\ell_{3}}{\ell_{4}}$ and $\vmodvec_{1}\in \virmod_{\ell_{1}}$, $\vmodvec_{2}\in\virmod_{\ell_{2}}$, $\vmodvec_{3}\in\virmod_{\ell_{3}}$, the formal series
 \begin{gather*}
 \voaint_{m \ell_{3}}^{\ell_{4}}\left(\voaint_{\ell_{1} \ell_{2}}^{m}(\vmodvec_{1},x_{0})\vmodvec_{2},x_{2}\right)\vmodvec_{3}
 \end{gather*}
 in $x_{0}$ and $x_{2}$ converges in $\ol{\virmod}_{\ell_{4}}$ at $x_{0}={\rm e}^{\log (z_{0})}$, $x_{2}={\rm e}^{\log (z_{2})}$ such that $|z_{2}|>|z_{0}|>0$.
\end{enumerate}
\end{thm}

From Theorem~\ref{thm:first-row_module_fusion}, we know that the fusion rules among first-row modules match those of finite-dimensional irreducible representations of $\Uq$.
The Clebsch--Gordan rules are, however, independent of the parameter $q$.
The following theorem manifests the matching of the parameters for the generic Virasoro VOA and the quantum group $\Uq$.

\begin{thm}[\cite{KoshidaKytola2022}]
\label{thm:associativity_int_op_frmods}
Let $\ell_{1},\ell_{2},\ell_{3},\ell_{4}\in\bZ_{\geq 0}$ and take $z_{1}$, $z_{2}$ on the real axis so that $z_{1}>z_{2}>z_{1}-z_{2}>0$.
For any $n\in \Chan{\ell_{2}}{\ell_{3}}{\ell_{1}}{\ell_{4}}$ and $\vmodvec_{1}\in\virmod_{\ell_{1}}$, $\vmodvec_{2}\in\virmod_{\ell_{2}}$, $\vmodvec_{3}\in \virmod_{\ell_{3}}$,
we get
\begin{gather*}
 \voaint_{\ell_{1} n}^{\ell_{4}}(\vmodvec_{1},z_{1})\voaint_{\ell_{2} \ell_{3}}^{n}(\vmodvec_{2},z_{2})\vmodvec_{3} \\
 \qquad = \sum_{m\in \Chan{\ell_{1}}{\ell_{2}}{\ell_{3}}{\ell_{4}}}\sixjsymb{\ell_{1}}{\ell_{2}}{\ell_{3}}{\ell_{4}}{m}{n}\voaint_{m \ell_{3}}^{\ell_{4}}\left(\voaint_{\ell_{1} \ell_{2}}^{m}(\vmodvec_{1},z_{1}-z_{2})\vmodvec_{2},z_{2}\right)\vmodvec_{3}
\end{gather*}
in $\ol{\virmod}_{\ell_{4}}$,
where the $6j$-symbols in the right hand side are those of $\Uq$ equipped with $\QGcoprod^{\op}$ defined in \eqref{eq:QG_6j_definition} at
$q={\rm e}^{\pi \bbmi t}$, $ t\not\in\bQ$.
\end{thm}

\section{Category of the first-row modules}
\label{sect:first-row_module_category}
In this section, we introduce the category $\catVir^{+}(t)$ of the first-row modules of the generic Virasoro VOA $\virvoa_{c}$,
and establish the ribbon tensor equivalence $\catVir^{+}(t)\to \catQG (q)$.

\subsection{Dual and opposite categories}
We first make tiny preliminaries on dual and opposite categories.
Let $\catgen$ be a category.
The dual category $\catgen^{\vee}$ is the category with the same objects as $\catgen$ and the opposite morphisms.
If $\catgen$ is equipped with a monoidal structure, then $\catgen^{\vee}$ naturally becomes a monoidal category.

Next, suppose that $\catgen$ is a monoidal category with the monoidal bifunctor denoted by $\otimes$.
The opposite category $\catgen^{\op}$ is the monoidal category with the same underlying category as $\catgen$ and the opposite monoidal structure $\otimes^{\op}$
$
 X\otimes^{\op}Y:=Y\otimes X$, $ X,Y\in\catgen$.

In the case where $\catgen$ is rigid, the dual and opposite categories are equivalent under a functor such that $X\mapsto X^{*}$ \cite[Chapter 2]{etingof2015tensor}.

\begin{rem}
Sometimes the dual category is called opposite, and the opposite category above is called reversed.
However, we follow the terminology in~\cite{etingof2015tensor}.
\end{rem}

When we apply the above construction to $\catgen=\catQG^{\op}(q)$, which we have already noticed is rigid, we can say that $(\catQG^{\op}(q))^{\vee}$ and $(\catQG^{\op}(q))^{\op}$ are equivalent.
Let us also see that the category~$(\catQG^{\op}(q))^{\op}$ is equivalent to $\catQG (q)$ as a tensor category.
Recall that $(\catQG^{\op}(q))^{\op}$ and $\catQG(q)$ have the same underlying abelian category $\catQG$.
We take the natural isomorphism
\begin{gather*}
 \Pi_{\QGrepU,\QGrepV}:=P_{\QGrepV,\QGrepU}\colon\ \QGrepU\otimes_{\QGcoprod^{\op}}^{\op}\QGrepV= \QGrepV\otimes_{\QGcoprod^{\op}}\QGrepU\to \QGrepU\otimes_{\QGcoprod}\QGrepV
\end{gather*}
to be the permutation for each $\QGrepU,\QGrepV\in \catQG$.
Indeed, $\Pi_{\QGrepU,\QGrepV}$ lives in $\Hom_{\catQG}(\QGrepV\otimes_{\QGcoprod^{\op}}\QGrepU,\QGrepU\otimes_{\QGcoprod}\QGrepV)$ as is checked as
\begin{align*}
 \QGcoprod (a)P_{\QGrepV,\QGrepU}(v\otimes u)&=\sum_{(a)}a_{(1)}u\otimes a_{(2)}v =P_{\QGrepV,\QGrepU}\Bigl(\sum_{(a)}a_{(2)}v\otimes a_{(1)}u\Bigr) \\
 &=P_{\QGrepV,\QGrepU}\left(\QGcoprod^{\op}(a)(v\otimes u)\right),\qquad v\in \QGrepV,\quad u\in \QGrepU,\quad a\in \Uq.
\end{align*}
Here, we wrote \smash{$\QGcoprod (a)=\sum_{(a)}a_{(1)}\otimes a_{(2)}$} for $a\in \Uq$.
Therefore, the identity functor \smash{$\Id_{\catQG}$} together with the natural isomorphisms $(\Pi_{\QGrepU,\QGrepV})_{\QGrepU,\QGrepV\in \catQG}$ defines an equivalence between $(\catQG^{\op}(q))^{\op}$ and~$\catQG (q)$ as tensor categories.

In conclusion, we may identify $(\catQG^{\op}(q))^{\vee}$, $(\catQG^{\op}(q))^{\op}$, and $\catQG(q)$ altogether as tensor categories.

\subsection[The first-row category catVir+]{The first-row category $\boldsymbol{\catVir^{+}}$}
We define the category $\catVir^{+}$ as the full subcategory of the module category of the generic Virasoro VOA $\virvoa_{c}$
generated by the first row modules $\virmod_{\ell}$, $\ell\in\bZ_{\geq 0}$ as an additive category.
Therefore, any object of $\catVir^{+}$ is isomorphic to a finite direct sum of first row modules,
and the morphism spaces are determined by
\begin{gather*}
 \Hom_{\catVir^{+}}(\virmod_{\ell_{1}},\virmod_{\ell_{2}})=
 \begin{cases}
 \bC \id_{\virmod_{\ell_{1}}},& \ell_{1}=\ell_{2},\\
 0, &\mathrm{otherwise}
 \end{cases}
\end{gather*}
for $\ell_{1},\ell_{2}\in\bZ_{\geq 0}$.
Note that, at this point, the category $\catVir^{+}$ is independent of the parameter $t$.

It is clear that $\catVir^{+}$ is equivalent to $\catQG$ as an abelian category.
Nevertheless, we would like to make equivalence functors explicit for later use.
For each object $\virmodgen\in\catVir^{+}$, we fix an isomorphism~${f_{\virmodgen}\colon \virmodgen\to\bigoplus_{\ell=0}^{\infty}\virmod_{\ell}^{\oplus m_{\ell}}}$.
Then, we define a functor $F\colon \catVir^{+}\to\catQG$ as follows.
At the object level, if $\virmodgen\in \catVir^{+}$ is isomorphic to $\bigoplus_{\ell=0}^{\infty}\virmod_{\ell}^{\oplus m_{\ell}}$, we send
\begin{gather*}
 F\colon\ \virmodgen\mapsto \bigoplus_{\ell=0}^{\infty}\QGirrep{\ell}^{\oplus m_{\ell}}.
\end{gather*}
At the morphism level, we first require that
\begin{gather*}
 F\colon\ \Hom_{\catVir^{+}}(\virmod_{\ell},\virmod_{\ell})\to \Hom_{\catQG}(\QGirrep{\ell},\QGirrep{\ell});\qquad \id_{\virmod_{\ell}}\mapsto \id_{\QGirrep{\ell}}
\end{gather*}
for each $\ell\in\bZ_{\geq 0}$.
For a general $h\colon \virmodgen_{1}\to\virmodgen_{2}$,
there is a unique way to send \smash{$f_{\virmodgen_{2}}\circ h\circ f_{\virmodgen_{1}}^{-1}$} so that $F$ induces linear maps on morphism spaces.
Then, we can simply define \smash{$F(h)=F\big(f_{\virmodgen_{2}}\circ h\circ f_{\virmodgen_{1}}^{-1}\big)$}.

Similarly, on the side of $\catQG$, we fix an isomorphism \smash{$g_{\QGrepU}\colon \QGrepU\to \bigoplus_{\ell=0}^{\infty}\QGirrep{\ell}^{\oplus n_{\ell}}$} for each \smash{$\QGrepU\in\catQG$}.
Then, we can define a functor $G\colon \catQG\to\catVir^{+}$ in the exactly analogous way as defining $F$.

Now, let us observe that the composition \smash{$G\circ F\colon \catVir^{+}\to\catVir^{+}$} is isomorphic to \smash{$\id_{\catVir^{+}}$}.
Indeed, the family \smash{$(f_{\virmodgen})_{\virmodgen\in\catVir^{+}}$} of the fixed isomorphisms gives a natural isomorphism \smash{$\id_{\catVir^{+}}\Rightarrow G\circ F $};
for~any $\virmodgen_{1},\virmodgen_{2}\!\in\! \catVir^{+}$ and $h\colon \virmodgen_{1}\to\virmodgen_{2}$,
the commutativity of the diagram
\begin{gather*}
 \xymatrix{
 \virmodgen_{1} \ar[r]^{f_{\virmodgen_{1}}} \ar[d]_{h} & G\circ F(\virmodgen_{1}) \ar[d]^{G\circ F(h)} \\
 \virmodgen_{2} \ar[r]_{f_{\virmodgen_{2}}} & G\circ F(\virmodgen_{2})
 }
\end{gather*}
follows from the definitions of $F$ and $G$.
Similarly, $F\circ G\colon \catQG\to \catQG$ is shown to be isomorphic to $\id_{\catQG}$.

In the rest of this section, we will see that $\catVir^{+}$ can be equipped with the structure of a~ribbon tensor category
following the general framework sketched in Section~\ref{sect:module_cat_VOA},
and will write the resulting ribbon tensor category as $\catVir^{+}(t)$.
At each step, we will compare the corresponding structure of~$\catVir^{+}(t)$ with that of $\catQG (q)$,
and thereby prove Theorem~\ref{thm:first-row-cat_QG}.

\subsection{Tensor structure}
First, we show that the category $\catVir^{+}$ is closed under the $P(z)$-tensor product for any $z\in \bC^{\times}$.
The following formula \eqref{eq:Pz-tensor_frmods} has been recorded in \cite[Theorem 5.2.2]{CJORY-tensor_categories_arising_from_Virasoro_algebra}, but we give a proof of it to keep the text elementary.

\begin{thm}
For any $z\in \bC^{\times}$, the category $\catVir^{+}$ is closed under the $P(z)$-tensor product.
Furthermore, for $\ell_{1},\ell_{2}\in\bZ_{\geq 0}$, the $P(z)$-tensor product $\virmod_{\ell_{1}}\boxtimes_{P(z)}\virmod_{\ell_{2}}$ is given by
\begin{gather}
\label{eq:Pz-tensor_frmods}
 \virmod_{\ell_{1}}\boxtimes_{P(z)}\virmod_{\ell_{2}}=\bigoplus_{\ell\in\Sel (\ell_{1},\ell_{2})}\virmod_{\ell}
\end{gather}
together with the $P(z)$-intertwining map
\begin{gather}
\label{eq:Pz-intmap_frmods}
 \voaint_{\virmod_{\ell_{1}}\boxtimes_{P(z)}\virmod_{\ell_{2}}}(-,z)=\sum_{\ell\in\Sel (\ell_{1},\ell_{2})}\voaint_{\ell_{1} \ell_{2}}^{\ell}(-,z)-.
\end{gather}
\end{thm}
\begin{proof}
Since the category $\catVir^{+}$ is semi-simple, it suffices to show that the $P(z)$-tensor product of simple objects exists and is given by the formulas \eqref{eq:Pz-tensor_frmods} and \eqref{eq:Pz-intmap_frmods}.
Let $\virmodgen\in \catVir^{+}$ together with $\voaint_{\virmodgen} (-,z)-\colon \virmod_{\ell_{1}}\otimes\virmod_{\ell_{2}}\to\ol{\virmodgen}$ be a $P(z)$-product of $\virmod_{\ell_{1}}$ and $\virmod_{\ell_{2}}$.
The object $\virmodgen$ can be decomposed into a direct sum of simple objects
$
 \virmodgen\simeq \bigoplus_{i=1}^{m}\virmod_{k_{i}}$,
where $k_{i}\in\bZ_{\geq 0}$, $i=1,\dots, m$.
In other words, we can find a family of injections $\iota_{i}\colon \virmod_{k_{i}}\to \virmodgen$ and projections $p_{i}\colon \virmodgen\to\virmod_{k_{i}}$, $i=1,\dots, m$, such that
\[
 p_{i}\circ \iota_{j}=\delta_{i,j}\id_{\virmod_{k_{i}}},\qquad i,j=1,\dots,m,\qquad \sum_{i=1}^{m}\iota_{i}\circ p_{i}=\id_{\virmodgen}.
\]
Then, for each $i=1,\dots, m$, the composition
$\ol{p}_{i}\circ \voaint_{\virmodgen}(-,z)-$
is an intertwining map of type \smash{$\inttyp{\virmod_{\ell_{1}}}{\virmod_{\ell_{2}}}{\virmod_{k_{i}}}$}.
Hence, there exists a unique homomorphism
\begin{gather*}
 f_{i}\colon\ \bigoplus_{\ell\in\Sel(\ell_{1},\ell_{2})}\virmod_{\ell}\to\virmod_{k_{i}}
\end{gather*}
such that
\begin{gather*}
 \ol{p}_{i}\circ \voaint_{\virmodgen}(-,z)-=\ol{f}_{i}\circ \voaint_{\virmod_{\ell_{1}}\boxtimes_{P(z)}\virmod_{\ell_{2}}}(-,z)-.
\end{gather*}
Therefore, the sum $f=\sum_{i=1}^{m}\iota_{i}\circ f_{i}$ is a homomorphism such that
\[
 \voaint_{\virmodgen}(-,z)-=\sum_{i=1}^{m}\ol{\iota_{i}\circ p_{i}}\circ \voaint_{\virmodgen}(-,z)-=\ol{f}\circ \voaint_{\virmod_{\ell_{1}}\boxtimes_{P(z)}\virmod_{\ell_{2}}}(-,z)-.
\]

Next, suppose that
\begin{gather*}
	f' \colon\ \bigoplus_{\ell\in\Sel(\ell_{1},\ell_{2})}\virmod_{\ell}\to \virmodgen
\end{gather*}
is a homomorphism such that
\begin{gather*}
	\voaint_{\virmodgen}(-,z)-=\ol{f'}\circ \voaint_{\virmod_{\ell_{1}}\boxtimes_{P(z)}\virmod_{\ell_{2}}}(-,z)-.
\end{gather*}
For each $i=1,\dots, m$, post-composition of both sides with $\ol{p_{i}}$ along with the uniqueness of $f_{i}$
give us $f_{i} = p_{i}\circ f'$. Thus, we have
$
	f' = \sum_{i=1}^{m}\iota_{i}\circ f_{i} = f
$
proving the uniqueness of $f$.
\end{proof}

We write $\pzproj{z}{\ell_{1}}{\ell_{2}}{\ell}$ for the canonical projection from $\virmod_{\ell_{1}}\boxtimes_{P(z)}\virmod_{\ell_{2}}$ to $\virmod_{\ell}$ for $\ell\in \Sel (\ell_{1},\ell_{2})$ according to the realization \eqref{eq:Pz-tensor_frmods}.
Then, it is characterized by the property
\begin{gather*}
 \ol{\pzproj{z}{\ell_{1}}{\ell_{2}}{\ell}} (\vmodvec_{1}\boxtimes_{P(z)}\vmodvec_{2})=\voaint_{\ell_{1} \ell_{2}}^{\ell}(\vmodvec_{1},z)\vmodvec_{2},\qquad \vmodvec_{1}\in\virmod_{\ell_{1}}, \quad \vmodvec_{2}\in\virmod_{\ell_{2}}.
\end{gather*}

Given two points $z_{1},z_{2}\in \bC^{\times}$ and a path $\gamma$ from $z_{1}$ to $z_{2}$, we can define the parallel transport isomorphism $\cT_{\gamma}\colon \boxtimes_{P(z_{1})}\Rightarrow \boxtimes_{P(z_{2})}$ as we have explained in Section~\ref{sect:module_cat_VOA}.
Although the parallel transport depends on (the homotopy class of) the path $\gamma$, its action can be described in a simple way, especially in the case when $\gamma$ does not change the branch.

\begin{prop}
\label{prop:triviality_parallel_transport}
Let $z_{1},z_{2}\in \bC^{\times}$ and $\gamma$ be a path from $z_{1}$ to $z_{2}$, along which ${\rm e}^{l_{\gamma}(z_{1})}$ has the same angle as $z_{1}$ in $[0,2\pi)$.
Suppose that, for $\ell_{1},\ell_{2}\in \bZ_{\geq 0}$, the $P(z)$-tensor product $\virmod_{\ell_{1}}\boxtimes_{P(z)}\virmod_{\ell_{2}}$ is realized by the formulas \eqref{eq:Pz-tensor_frmods} and \eqref{eq:Pz-intmap_frmods}.
Then, the parallel transport $\cT_{\gamma}$ acts as $\id_{\virmod_{\ell}}$ on each component $\virmod_{\ell}$, $\ell\in \Sel(\ell_{1},\ell_{2})$ appearing in \eqref{eq:Pz-tensor_frmods}.
\end{prop}
\begin{proof}
In the realization in \eqref{eq:Pz-tensor_frmods} and \eqref{eq:Pz-intmap_frmods}, the $P(z)$-tensor product does not depend on $z$ at the object level, so we can think that $\virmod_{\ell_{1}}\boxtimes_{P(z_{1})}\virmod_{\ell_{2}}=\virmod_{\ell_{1}}\boxtimes_{P(z_{2})}\virmod_{\ell_{2}}$ as objects.
Recall that~$\cT_{\gamma}$ is characterized by the property
\begin{gather*}
 \ol{\cT}_{\gamma}\colon\ \sum_{\ell\in\Sel(\ell_{1},\ell_{2})}\voaint_{\ell_{1} \ell_{2}}^{\ell}(\vmodvec_{1},z_{1})\vmodvec_{2}\mapsto \sum_{\ell\in\Sel(\ell_{1},\ell_{2})}\voaint_{\ell_{1} \ell_{2}}^{\ell}\big(\vmodvec_{1},{\rm e}^{l_{\gamma}(z_{1})}\big)\vmodvec_{2}
\end{gather*}
for $\vmodvec_{1}\in\virmod_{\ell_{1}}, \vmodvec_{2}\in\virmod_{\ell_{2}}$,
but under the assumption on $\gamma$, we have
\begin{gather*}
 \voaint_{\ell_{1} \ell_{2}}^{\ell}\big(\vmodvec_{1},{\rm e}^{l_{\gamma}(z_{1})}\big)\vmodvec_{2}=\voaint_{\ell_{1} \ell_{2}}^{\ell}(\vmodvec_{1},z_{1})\vmodvec_{2},\qquad \ell\in \Sel (\ell_{1},\ell_{2}).
\end{gather*}
Therefore, $\cT_{\gamma}$ acts as $\id_{\virmod_{\ell}}$ on each $\virmod_{\ell}$, $\ell\in\Sel (\ell_{1},\ell_{2})$.
\end{proof}

Let us look into the associativity isomorphism on $\catVir^{+}$.
First, we clarify the structure of the composition and iteration of the tensor product in more detail, focusing our attention on simple objects.
We take $z_{1}$, $z_{2}$ on the real axis so that $z_{1}>z_{2}>z_{1}-z_{2}>0$ and the paths $\gamma_{1}$, $\gamma_{2}$, $\gamma_{3}$,~$\gamma_{4}$ as before.
Let us fix $\ell_{1},\ell_{2},\ell_{3}\in\bZ_{\geq 0}$.
As we have already noticed, according to the realization of the $P(z)$-tensor product as in \eqref{eq:Pz-tensor_frmods} and \eqref{eq:Pz-intmap_frmods}, we can identify
\begin{gather}
\label{eq:composition_frmods}
 \virmod_{\ell_{1}}\boxtimes_{P(z_{1})} (\virmod_{\ell_{2}}\boxtimes_{P(z_{2})} \virmod_{\ell_{3}})
 =\virmod_{\ell_{1}}\boxtimes (\virmod_{\ell_{2}}\boxtimes \virmod_{\ell_{3}})=\bigoplus_{\ell_{\infty}\in\bZ_{\geq 0}}\bigoplus_{n\in \Chan{\ell_{2}}{\ell_{3}}{\ell_{1}}{\ell_{\infty}}}(\virmod_{\ell_{\infty}})^{(n)},
\end{gather}
where $(\virmod_{\ell_{\infty}})^{(n)}$ are copies of $\virmod_{\ell_{\infty}}$ and the canonical projection to each $(\virmod_{\ell_{\infty}})^{(n)}$ is given by~\smash{$
 \pzproj{z_{1}}{\ell_{1}}{n}{\ell_{\infty}}\circ \big(\id_{\virmod_{\ell_{1}}}\boxtimes_{P(z_{1})}\pzproj{z_{2}}{\ell_{2}}{\ell_{3}}{n}\big)
$}.
From the definition, it is clear that
\begin{gather*}
 \ol{\pzproj{z_{1}}{\ell_{1}}{n}{\ell_{\infty}}\circ \big(\id_{\virmod_{\ell_{1}}}\boxtimes_{P(z_{1})}\pzproj{z_{2}}{\ell_{2}}{\ell_{3}}{n}\big)}\colon \\
 \qquad \vmodvec_{1}\boxtimes_{P(z_{1})}(\vmodvec_{2}\boxtimes_{P(z_{2})}\vmodvec_{3}) \mapsto \voaint_{\ell_{1} n}^{\ell_{\infty}}(\vmodvec_{1},z_{1})\voaint_{\ell_{2} \ell_{3}}^{n}(\vmodvec_{2},z_{2})\vmodvec_{3}
\end{gather*}
for $\vmodvec_{i}\in\virmod_{\ell_{i}}$, $i=1,2,3$.
Furthermore, from Proposition~\ref{prop:triviality_parallel_transport}, the composition of parallel transports $\cT_{\gamma_{1}}\circ (\id_{\virmod_{\ell_{1}}}\boxtimes_{P(z_{1})} \cT_{\gamma_{2}})$ acts as $\id_{(\virmod_{\ell_{\infty}})^{(n)}}$ on each component $(\virmod_{\ell_{\infty}})^{(n)}$, $\ell_{\infty}\in\bZ_{\geq 0}$, $n\in \Chan{\ell_{2}}{\ell_{3}}{\ell_{1}}{\ell_{\infty}}$.

Similarly, for the iteration, we have the identification
\begin{gather}
\label{eq:iteration_frmods}
 (\virmod_{\ell_{1}}\boxtimes_{P(z_{1}-z_{2})}\virmod_{\ell_{2}})\boxtimes_{P(z_{2})}\virmod_{\ell_{3}}
 =(\virmod_{\ell_{1}}\boxtimes \virmod_{\ell_{2}})\boxtimes\virmod_{\ell_{3}}
 =\bigoplus_{\ell_{\infty}\in\bZ_{\geq 0}}\bigoplus_{m\in I^{\ell_{1} \ell_{2}}_{\ell_{3} \ell_{\infty}}}(\virmod_{\ell_{\infty}})_{(m)},
\end{gather}
where each copy $(\virmod_{\ell_{\infty}})_{(m)}$ is the image of the projection
$
 \pzproj{z_{2}}{m}{\ell_{3}}{\ell_{\infty}}\circ \big(\pzproj{z_{1}-z_{2}}{\ell_{1}}{\ell_{2}}{m}\boxtimes_{P(z_{2})}\id_{\virmod_{\ell_{3}}}\big)$.
Again, from the definition, this projection is characterized by the property
\begin{gather*}
 \ol{\pzproj{z_{2}}{m}{\ell_{3}}{\ell_{\infty}}\circ \big(\pzproj{z_{1}-z_{2}}{\ell_{1}}{\ell_{2}}{m}\boxtimes_{P(z_{2})}\id_{\virmod_{\ell_{3}}}\big)}\colon \\
 \qquad(\vmodvec_{1}\boxtimes_{P(z_{1}-z_{2})}\vmodvec_{2})\boxtimes_{P(z_{2})}\vmodvec_{3} \mapsto \voaint_{m \ell_{3}}^{\ell_{\infty}}(\voaint_{\ell_{1} \ell_{2}}^{m}(\vmodvec_{1},z_{1}-z_{2})\vmodvec_{2},z_{2})\vmodvec_{3}
\end{gather*}
for $\vmodvec_{i}\in\virmod_{\ell_{i}}$, $i=1,2,3$.
By Proposition~\ref{prop:triviality_parallel_transport}, the composition of parallel transports $(\cT_{\gamma_{4}}\boxtimes_{P(z_{2})}\id_{\virmod_{\ell_{3}}})\circ \cT_{\gamma_{3}}$ acts as $\id_{(\virmod_{\ell_{\infty}})_{(m)}}$ on each component $(\virmod_{\ell_{\infty}})_{(m)}$, $\ell_{\infty}\in \bZ_{\geq 0}$, \smash{$m\in \Chan{\ell_{1}}{\ell_{2}}{\ell_{3}}{\ell_{\infty}}$}.

The above observations allow us to conclude that the associativity isomorphism $\cA_{\virmod_{\ell_{1}},\virmod_{\ell_{2}},\virmod_{\ell_{3}}}$ coincides with the resolved version \smash{$\big(\cA_{P(z_{1}),P(z_{2})}^{P(z_{1}-z_{2}),P(z_{2})}\big)_{\virmod_{\ell_{1}},\virmod_{\ell_{2}},\virmod_{\ell_{3}}}$} under the identifications \eqref{eq:composition_frmods} and \eqref{eq:iteration_frmods} and amounts to a homomorphism of the form
\begin{gather*}
 \cA_{\virmod_{\ell_{1}},\virmod_{\ell_{2}},\virmod_{\ell_{3}}}=\sum_{\ell_{\infty}\in\bZ_{\geq 0}}\sum_{m\in \Chan{\ell_{1}}{\ell_{2}}{\ell_{3}}{\ell_{\infty}}}\sum_{n\in \Chan{\ell_{2}}{\ell_{3}}{\ell_{1}}{\ell_{\infty}}}\ascmtx{\ell_{1}}{\ell_{2}}{\ell_{3}}{\ell_{\infty}}{m}{n},\qquad \ascmtx{\ell_{1}}{\ell_{2}}{\ell_{3}}{\ell_{\infty}}{m}{n}\colon\ (\virmod_{\ell_{\infty}})_{(m)}\to (\virmod_{\ell_{\infty}})^{(n)}.
\end{gather*}

\begin{thm}
\label{thm:assoc_6j}
The associativity isomorphism $\cA_{\virmod_{\ell_{1}},\virmod_{\ell_{2}},\virmod_{\ell_{3}}}$ is given by
\begin{gather*}
 \ascmtx{\ell_{1}}{\ell_{2}}{\ell_{3}}{\ell_{\infty}}{m}{n}=\sixjsymb{\ell_{1}}{\ell_{2}}{\ell_{3}}{\ell_{\infty}}{m}{n}\id_{\virmod_{\ell_{\infty}}},\qquad \ell_{\infty}\in\bZ_{\geq 0}, \quad m\in \Chan{\ell_{1}}{\ell_{2}}{\ell_{3}}{\ell_{\infty}}, \quad n\in \Chan{\ell_{2}}{\ell_{3}}{\ell_{1}}{\ell_{\infty}}.
\end{gather*}
Here, the $6$j-symbols are those of $\Uq$ at \smash{$q={\rm e}^{\pi\bbmi t}$} equipped with the opposite coproduct $\QGcoprod^{\op}$.
\end{thm}
\begin{proof}
It suffices to show that the isomorphism $\cA_{\virmod_{\ell_{1}},\virmod_{\ell_{2}},\virmod_{\ell_{3}}}$ defined in such a way behaves as
\begin{gather*}
 \cA_{\virmod_{\ell_{1}},\virmod_{\ell_{2}},\virmod_{\ell_{3}}}\colon\ (\vmodvec_{1}\boxtimes_{P(z_{1}-z_{2})}\vmodvec_{2})\boxtimes_{P(z_{2})}\vmodvec_{3}\mapsto \vmodvec_{1}\boxtimes_{P(z_{1})}(\vmodvec_{2}\boxtimes_{P(z_{2})}\vmodvec_{3})
\end{gather*}
for $\vmodvec_{i}\in\virmod_{\ell_{i}}$, $i=1,2,3$.
Recall that
\begin{gather*}
 (\vmodvec_{1}\boxtimes_{P(z_{1}-z_{2})}\vmodvec_{2})\boxtimes_{P(z_{2})}\vmodvec_{3}=\sum_{\ell_{\infty}\in\bZ_{\geq 0}}\sum_{m\in \Chan{\ell_{1}}{\ell_{2}}{\ell_{3}}{\ell_{\infty}}}\voaint_{m \ell_{3}}^{\ell_{\infty}}(\voaint_{\ell_{1} \ell_{2}}^{m}(\vmodvec_{1},z_{1}-z_{2})\vmodvec_{2},z_{2})\vmodvec_{3},
\end{gather*}
which is sent by $\cA_{\virmod_{\ell_{1}},\virmod_{\ell_{2}},\virmod_{\ell_{3}}}$ to
\begin{align}
\label{eq:associativity_computation_intermediate}
 \sum_{\ell_{\infty}\in\bZ_{\geq 0}}\sum_{n\in \Chan{\ell_{2}}{\ell_{3}}{\ell_{1}}{\ell_{\infty}}}\bigg(\sum_{m\in \Chan{\ell_{1}}{\ell_{2}}{\ell_{3}}{\ell_{\infty}}}\sixjsymb{\ell_{1}}{\ell_{2}}{\ell_{3}}{\ell_{\infty}}{m}{n}\voaint_{m \ell_{3}}^{\ell_{\infty}}\bigg(\voaint_{\ell_{1} \ell_{2}}^{m}(\vmodvec_{1},z_{1}-z_{2})\vmodvec_{2},z_{2}\bigg)\vmodvec_{3}\bigg),
\end{align}
where each summand
\begin{gather*}
 \sum_{m\in \Chan{\ell_{1}}{\ell_{2}}{\ell_{3}}{\ell_{\infty}}}\sixjsymb{\ell_{1}}{\ell_{2}}{\ell_{3}}{\ell_{\infty}}{m}{n}\voaint_{m \ell_{3}}^{\ell_{\infty}}(\voaint_{\ell_{1} \ell_{2}}^{m}(\vmodvec_{1},z_{1}-z_{2})\vmodvec_{2},z_{2})\vmodvec_{3}
\end{gather*}
lives in $(\virmod_{\ell_{\infty}})^{(n)}$.
Here, we use Theorem~\ref{thm:associativity_int_op_frmods} to conclude that \eqref{eq:associativity_computation_intermediate} coincides with
\begin{align*}
 \sum_{\ell_{\infty}\in\bZ_{\geq 0}}\sum_{n\in \Chan{\ell_{2}}{\ell_{3}}{\ell_{1}}{\ell_{\infty}}}\voaint_{\ell_{1} n}^{\ell_{\infty}}(\vmodvec_{1},z_{1})\voaint_{\ell_{2} \ell_{3}}^{n}(\vmodvec_{2},z_{2})\vmodvec_{3}=\vmodvec_{1}\boxtimes_{P(z_{1})}(\vmodvec_{2}\boxtimes_{P(z_{2})}\vmodvec_{3}).
\end{align*}
Therefore, the isomorphism $\cA_{\virmod_{\ell_{1}},\virmod_{\ell_{2}},\virmod_{\ell_{3}}}$ satisfies the desired property.
\end{proof}

We can define unit isomorphisms $\lambda_{\virmodgen}\colon \virvoa_{c}\boxtimes\virmodgen\to\virmodgen$ and $\rho_{\virmodgen}\colon \virmodgen\boxtimes\virvoa_{c}\to \virmodgen$ for each $\virmodgen\in\catVir^{+}$ according to the general theory.
If $\virmodgen=\virmod_{\ell}$, $\ell\in\bZ_{\geq 0}$ is a simple object, we think that $\virvoa_{c}\boxtimes\virmod_{\ell}=\virmod_{\ell}=\virmod_{\ell}\boxtimes\virvoa_{c}$ under the realization in \eqref{eq:Pz-tensor_frmods} and \eqref{eq:Pz-intmap_frmods}.
\begin{thm}
\label{thm:unit_isom_catvir}
Let $\ell\in\bZ_{\geq 0}$. Under the realization in \eqref{eq:Pz-tensor_frmods} and \eqref{eq:Pz-intmap_frmods}, both $\lambda_{\virmod_{\ell}}\colon \virvoa_{c}\boxtimes \virmod_{\ell}=\virmod_{\ell}\to\virmod_{\ell}$ and $\rho_{\virmod_{\ell}}\colon \virmod_{\ell}\boxtimes\virvoa_{c}=\virmod_{\ell}\to\virmod_{\ell}$ are $\id_{\virmod_{\ell}}$.
\end{thm}
\begin{proof}
Recall that the left unit isomorphism $\lambda_{\virmod_{\ell}}$ is characterized by $\lambda_{\virmod_{\ell}}(\voavac\boxtimes \vmodvec)=\vmodvec$, $\vmodvec\in\virmod_{\ell}$.
From \eqref{eq:Pz-intmap_frmods}, we have
\smash{$
 \voavac\boxtimes \vmodvec=\voaint_{0 \ell}^{\ell}(\voavac,1)\vmodvec=\vmodvec$}.
Here we used $\voaint_{0 \ell}^{\ell}(\voavac,x)=\id$, which we have noticed in Remark~\ref{rem:intertw_Virasoro_special_cases}.
Hence, we observe $\lambda_{\virmod_{\ell}}=\id_{\virmod_{\ell}}$.
As for the right unit isomorphism, it is characterized by \smash{$\ol{\rho_{\virmod_{\ell}}}(\vmodvec\boxtimes \voavac)={\rm e}^{L^{\virmod_{\ell}}_{-1}}\vmodvec$}, $\vmodvec\in\virmod_{\ell}$.
Again, from \eqref{eq:Pz-intmap_frmods} and the property of $\voaint_{\ell 0}^{\ell}(-,x)$ pointed out in Remark~\ref{rem:intertw_Virasoro_special_cases}, we see that
\begin{gather*}
 \vmodvec\boxtimes \voavac = \voaint_{\ell 0}^{\ell}(\vmodvec,1)={\rm e}^{L^{\virmod_{\ell}}_{-1}}\vmodvec,
\end{gather*}
which verifies that $\rho_{\virmod_{\ell}}=\id_{\virmod_{\ell}}$.
\end{proof}

As for the duality structure, each simple object $\virmod_{\ell}$, $\ell\in \bZ_{\geq 0}$ is self-dual, and the tensor product $\virmod_{\ell}\boxtimes\virmod_{\ell}$ contains a unique component $\virmod_{0}\simeq \virvoa_{c}$.
We define the evaluation morphisms
\begin{align*}
	\ev_{\virmod_{\ell}}\colon\ \virmod_{\ell}^{\vee}\boxtimes \virmod_{\ell} \to \virvoa_{c},\qquad
	\ev_{\virmod_{\ell}}'\colon\ \virmod_{\ell}\boxtimes \virmod_{\ell}^{\vee} \to \virvoa_{c}
\end{align*}
and the coevaluation morphisms
\begin{align*}
	\coev_{\virmod_{\ell}}\colon\ \virvoa_{c} \to \virmod_{\ell}\boxtimes \virmod_{\ell}^{\vee},\qquad
	\coev_{\virmod_{\ell}}'\colon\ \virvoa_{c} \to \virmod_{\ell}^{\vee}\boxtimes \virmod_{\ell}
\end{align*}to be the canonical projection
$\virmod_{\ell}\boxtimes\virmod_{\ell} \to \virvoa_{c}$ and the canonical injection
$\virvoa_{c} \to \virmod_{\ell}\boxtimes\virmod_{\ell}$.

Notice that, as an abelian category, $\catVir^{+}$ is independent of the parameter $t$.
As is clear from Theorem~\ref{thm:assoc_6j}, the associativity isomorphisms $\cA$ depend on $t$.

\begin{defn}
We write the category $\catVir^{+}$ equipped with $\cA$, $\lambda$, $\rho$,
$\ev$, $\ev'$, $\coev$, and $\coev'$ defined above as $\catVir^{+}(t)$.
\end{defn}

Let us start comparing the structure of $\catVir^{+}(t)$ to that of $\catQG(q)$.
For that, we first define the natural isomorphism
$
	J \colon F(- \boxtimes -) \Rightarrow F(-)\otimes F(-)
$
on simple objects by
\begin{gather*}
	J_{\virmod_{\ell_{1}},\virmod_{\ell_{2}}} = \sum_{\ell\in \Sel (\ell_{1},\ell_{2})}\CGemb{\ell}{\ell_{1}}{\ell_{2}},\qquad \ell_{1},\ell_{2}\in \bZ_{\geq 0}.
\end{gather*}

\begin{thm}
\label{thm:equivalence_vir_QG}
Assume the parameter matching $q={\rm e}^{\pi\bbmi t}$.
Under the identification of objects in~$\catQG (q)$ by $J$, the functor $F\colon \catVir^{+}(t) \to \catQG(q)$ maps
\begin{gather*}
	F(\cA_{\virmod_{\ell_{1}},\virmod_{\ell_{2}},\virmod_{\ell_{3}}}) = \alpha_{\QGirrep{\ell_{1}},\QGirrep{\ell_{2}},\QGirrep{\ell_{3}}}, \qquad
	F(\lambda_{\virmod_{\ell}}) = F(\rho_{\virmod_{\ell}}) = \id_{\QGirrep{\ell}}, \\
	F(\ev_{\virmod_{\ell}}) = F(\ev_{\virmod_{\ell}}') = \CGproj{0}{\ell}{\ell},\qquad
	F(\coev_{\virmod_{\ell}}) = F(\coev_{\virmod_{\ell}}') = \CGemb{0}{\ell}{\ell}
\end{gather*}
for $\ell_{1},\ell_{2},\ell_{3},\ell\in \bZ_{\geq 0}$.
Therefore, $\catVir^{+}(t)$ is a tensor category and the pair $(F,J)$ is an equivalence of tensor categories.
\end{thm}
\begin{proof}
The claim about $\lambda$, $\rho$ follows from Theorem~\ref{thm:unit_isom_catvir}.
As for the evaluation and coevaluation morphisms, the asserted properties follows from definition.
It remains to show the coincidence of associativity isomorphisms.

Recall the equivalence \smash{$\catQG (q) \simeq (\catQG^{\op}(q))^{\vee}$} of tensor categories.
We compare the associativity of $\catVir^{+}(t)$ with that of $(\catQG^{\op}(q))^{\vee}$ instead of $\catQG(q)$.
The associativity isomorphism $\cA_{\virmod_{\ell_{1}},\virmod_{\ell_{2}},\virmod_{\ell_{3}}}$ of \smash{$\catVir^{+}$} is characterized by the property
\begin{gather*}
 \pzproj{z_{1}}{\ell_{1}}{n}{\ell_{\infty}}\circ (\id_{\virmod_{\ell_{1}}}\boxtimes_{P(z_{1})}\pzproj{z_{2}}{\ell_{2}}{\ell_{3}}{n})\circ \cA_{\virmod_{\ell_{1}},\virmod_{\ell_{2}},\virmod_{\ell_{3}}} \\
 \qquad = \sum_{m\in \Chan{\ell_{1}}{\ell_{2}}{\ell_{3}}{\ell_{\infty}}}\sixjsymb{\ell_{1}}{\ell_{2}}{\ell_{3}}{\ell_{\infty}}{m}{n}\pzproj{z_{2}}{m}{\ell_{3}}{\ell_{\infty}}\circ (\pzproj{z_{1}-z_{2}}{\ell_{1}}{\ell_{2}}{m}\boxtimes_{P(z_{2})}\id_{\virmod_{\ell_{3}}}).
\end{gather*}
for $n\in \Chan{\ell_{2}}{\ell_{3}}{\ell_{1}}{\ell_{\infty}}$.
We compare this with \eqref{eq:QG_6j_definition} to conclude that
\begin{align*}
 F(\cA_{\virmod_{\ell_{1}},\virmod_{\ell_{2}},\virmod_{\ell_{3}}})&= (\alpha_{\QGirrep{\ell_{1}},\QGirrep{\ell_{2}},\QGirrep{\ell_{3}}}^{\op})^{-1}\\
 &\in \Hom_{\catQG^{\op}(q)}(\QGirrep{\ell_{1}}\otimes_{\QGcoprod^{\op}} (\QGirrep{\ell_{2}}\otimes_{\QGcoprod^{\op}}\QGirrep{\ell_{3}}),(\QGirrep{\ell_{1}}\otimes_{\QGcoprod^{\op}}\QGirrep{\ell_{2}})\otimes_{\QGcoprod^{\op}}\QGirrep{\ell_{3}}) \\
 &= \Hom_{(\catQG^{\op}(q))^{\vee}}((\QGirrep{\ell_{1}}\otimes_{\QGcoprod^{\op}}\QGirrep{\ell_{2}})\otimes_{\QGcoprod^{\op}}\QGirrep{\ell_{3}},\QGirrep{\ell_{1}}\otimes_{\QGcoprod^{\op}} (\QGirrep{\ell_{2}}\otimes_{\QGcoprod^{\op}}\QGirrep{\ell_{3}})).
\end{align*}
This proves the desired result.
\end{proof}

\begin{rem}
We make a comment on the reason why we compared the associativity of $\catVir^{+}(t)$ with that of $(\catQG^{\op}(q))^{\vee}$, but not $\catQG(q)$.
Recall that our $6j$-symbols are defined as the matrix elements in terms of injections.
On the VOA side, the same $6j$-symbols naturally appear as matrix elements in projections.
Therefore, to match injections to projections, one needs to take the dual of either category.
This also explains why we defined the $6j$ symbols associated with the opposite coproduct $\QGcoprod^{\op}$.
Since we have decided to take the dual category on the quantum group side, we can eventually return to the original coproduct $\QGcoprod$ under the equivalence~${(\catQG^{\op}(q))^{\vee}\simeq \catQG(q)}$.
\end{rem}

\subsection{Braiding}
We move on to comparing the structures of braiding on $\catVir^{+}(t)$ and $\catQG(q)$.
We first calculate the braiding of $\catVir^{+}(t)$ following the definition given in Section~\ref{sect:module_cat_VOA}.
Under the realization \eqref{eq:Pz-tensor_frmods}, we~can consider $\virmod_{\ell_{1}}\boxtimes \virmod_{\ell_{2}}$ and $\virmod_{\ell_{2}}\boxtimes\virmod_{\ell_{1}}$ to be identical as objects.
\begin{prop}
\label{prop:braiding_first_row}
Let $\ell_{1},\ell_{2}\in \bZ_{\geq 0}$.
Under the realization of $\virmod_{\ell_{1}}\boxtimes \virmod_{\ell_{2}}$ and $\virmod_{\ell_{2}}\boxtimes\virmod_{\ell_{1}}$ given by the formulas \eqref{eq:Pz-tensor_frmods} and \eqref{eq:Pz-intmap_frmods}, the braiding isomorphism $c_{\virmod_{\ell_{1}},\virmod_{\ell_{2}}}$ is given by
\begin{gather*}
 c_{\virmod_{\ell_{1}},\virmod_{\ell_{2}}}=\sum_{\ell\in \Sel (\ell_{1},\ell_{2})}{\rm e}^{\pi\bbmi (\frweight{\ell}-\frweight{\ell_{1}}-\frweight{\ell_{2}})}\id_{\virmod_{\ell}}.
\end{gather*}
\end{prop}
\begin{proof}
We compare $\vmodvec_{1}\boxtimes\vmodvec_{2}$ and \smash{${\rm e}^{L^{\virmod_{\ell_{2}}\boxtimes\virmod_{\ell_{1}}}_{-1}}\ol{\cT_{\gamma}}(\vmodvec_{2}\boxtimes_{P(-1)}\vmodvec_{1})$} for $\vmodvec_{1}\!\in\!\virmod_{\ell_{1}}$, $\vmodvec_{2}\!\in\!\virmod_{\ell_{2}}$.
From~\eqref{eq:Pz-intmap_frmods} we have
\begin{gather*}
 \vmodvec_{1}\boxtimes\vmodvec_{2}=\sum_{\ell\in\Sel (\ell_{1},\ell_{2})}\voaint_{\ell_{1}\ell_{2}}^{\ell}(\vmodvec_{1},1)\vmodvec_{2}
\end{gather*}
and on the other hand, we can see that
\begin{gather*}
 {\rm e}^{L^{\virmod_{\ell_{2}}\boxtimes\virmod_{\ell_{1}}}_{-1}}\ol{\cT_{\gamma}}(\vmodvec_{2}\boxtimes_{P(-1)}\vmodvec_{1})=\sum_{\ell\in\Sel (\ell_{1},\ell_{2})}\big(\Omega\voaint_{\ell_{2}\ell_{1}}^{\ell}\big)(\vmodvec_{1},1)\vmodvec_{2}.
\end{gather*}
Recall the transformation \eqref{eq:int_flip} of intertwining operators.
For each $\ell\in\bZ_{\geq 0}$, we have
\begin{gather*}
\big(\Omega\voaint_{\ell_{2}\ell_{1}}^{\ell}\big)(-,x)\in\intsp{\virmod_{\ell_{1}}}{\virmod_{\ell_{2}}}{\virmod_{\ell}}.
\end{gather*}
In particular, it must be proportional to $\voaint_{\ell_{1}\ell_{2}}^{\ell}(-,x)$ since the corresponding space of intertwining operators is one-dimensional.
We can fix the constant of proportionality by looking at the image of the highest weight vectors
\begin{gather*}
 \big(\Omega\voaint_{\ell_{2}\ell_{1}}^{\ell}\big)(\hwvec_{\ell_{1}},x)\hwvec_{\ell_{2}}\in
 \fusNorm_{\ell_{2} \ell_{1}}^{\ell}{\rm e}^{\pi\bbmi (\frweight{\ell}-\frweight{\ell_{1}}-\frweight{\ell_{2}})}\hwvec_{\ell}x^{\frweight{\ell}-\frweight{\ell_{1}}-\frweight{\ell_{2}}}+\virmod_{\ell}[[x]]x^{\frweight{\ell}-\frweight{\ell_{1}}-\frweight{\ell_{2}}+1}.
\end{gather*}
Since $\fusNorm_{\ell_{1} \ell_{2}}^{\ell}=\fusNorm_{\ell_{2} \ell_{1}}^{\ell}$, we can see that
\smash{$
 \big(\Omega\voaint_{\ell_{2}\ell_{1}}^{\ell}\big)(-,x)={\rm e}^{\pi\bbmi (\frweight{\ell}-\frweight{\ell_{1}}-\frweight{\ell_{2}})}\voaint_{\ell_{1}\ell_{2}}^{\ell}(-,x)$}.
Therefore, we~obtain the formula
\begin{gather*}
 {\rm e}^{L^{\virmod_{\ell_{2}}\boxtimes\virmod_{\ell_{1}}}_{-1}}\ol{\cT_{\gamma}}\left(\vmodvec_{2}\boxtimes_{P(-1)}\vmodvec_{1}\right)=\sum_{\ell\in\Sel (\ell_{1},\ell_{2})}{\rm e}^{\pi\bbmi (\frweight{\ell}-\frweight{\ell_{1}}-\frweight{\ell_{2}})}\voaint_{\ell_{1}\ell_{2}}^{\ell}(\vmodvec_{1},1)\vmodvec_{2},
\end{gather*}
which concludes the desired result.
\end{proof}

We now see that the braiding in the above theorem matches that on $\catQG(q)$.

\begin{thm}
Along the tensor functor $(F,J) \colon \catVir^{+}(t)\to\catQG(q)$,
$
 F(c_{\virmod_{\ell_{1}},\virmod_{\ell_{2}}})=\braid_{\QGirrep{\ell_{1}},\QGirrep{\ell_{2}}}
$
for each $\ell_{1},\ell_{2}\in\bZ_{\geq 0}$.
\end{thm}
\begin{proof}
It suffices to show that
\[
 \CGproj{\ell}{\ell_{2}}{\ell_{1}}\circ \braid_{\QGirrep{\ell_{1}},\QGirrep{\ell_{2}}}\circ \CGemb{\ell}{\ell_{1}}{\ell_{2}}={\rm e}^{\pi \bbmi (\frweight{\ell}-\frweight{\ell_{1}}-\frweight{\ell_{2}})}\id_{\virmod_{\ell}}
\]
for each $\ell\in\Sel(\ell_{1},\allowbreak\ell_{2})$.
Let us apply $\braid_{\QGirrep{\ell_{1}},\QGirrep{\ell_{2}}}$ on \smash{$\CGemb{\ell}{\ell_{1}}{\ell_{2}}\big(\QGirvec{\ell}{0}\big)$} to find
\begin{gather*}
 \braid_{\QGirrep{\ell_{1}},\QGirrep{\ell_{2}}}\circ \CGemb{\ell}{\ell_{1}}{\ell_{2}}\big(\QGirvec{\ell}{0}\big)
 \in \CGcoeff{\ell}{\ell_{1}}{\ell_{2}}{0}q^{\frac{1}{2}\ell_{1}(\ell_{2}-2s)}\QGirvec{\ell_{2}}{s}\otimes\QGirvec{\ell_{1}}{0}+\sum_{j=1}^{s}\bC \QGirvec{\ell_{2}}{s-j}\otimes\QGirvec{\ell_{1}}{j}.
\end{gather*}
Here we set $s=(\ell_{1}+\ell_{2}-\ell)/2$ as before.
Since the composition \smash{$\CGproj{\ell}{\ell_{2}}{\ell_{1}}\circ \braid_{\QGirrep{\ell_{1}},\QGirrep{\ell_{2}}}\circ \CGemb{\ell}{\ell_{1}}{\ell_{2}}$} must be proportional to the identity,
we have
\begin{align*}
 \CGproj{\ell}{\ell_{2}}{\ell_{1}}\circ \braid_{\QGirrep{\ell_{1}},\QGirrep{\ell_{2}}}\circ \CGemb{\ell}{\ell_{1}}{\ell_{2}}=\frac{\CGcoeff{\ell}{\ell_{1}}{\ell_{2}}{0}}{\CGcoeff{\ell}{\ell_{2}}{\ell_{1}}{s}}q^{\frac{1}{2}\ell_{1}(\ell_{2}-2s)}\id_{\QGirrep{\ell}}.
\end{align*}
Substituting the explicit formula for the Clebsch--Gordan coefficients \eqref{eq:CG_coeffs}, the constant becomes
\begin{gather*}
 \frac{\CGcoeff{\ell}{\ell_{1}}{\ell_{2}}{0}}{\CGcoeff{\ell}{\ell_{2}}{\ell_{1}}{s}}q^{\frac{1}{2}\ell_{1}(\ell_{2}-2s)}
 =(-1)^{s}q^{\frac{1}{2}\ell_{1}\ell_{2}-s(\ell_{1}+\ell_{2})+s^{2}-s}.
\end{gather*}
When we recall that $q$ is related to $t$ by $q={\rm e}^{\pi\bbmi t}$ and use the formula \eqref{eq:conf_weight_fr}, we can check that this constant coincides with ${\rm e}^{\pi\bbmi (\frweight{\ell}-\frweight{\ell_{1}}-\frweight{\ell_{2}})}$ to complete the proof.
\end{proof}

\subsection{Ribbon structure}
Finally, we compare the ribbon structures.
Let us state the result immediately.
\begin{thm}
Along the functor $(F,J) \colon \catVir^{+}(t)\to\catQG(q)$, we have
$
 F(\theta_{\virmod_{\ell}})=\theta_{\QGirrep{\ell}}
$
for each $\ell\in\bZ_{\geq 0}$.
\end{thm}
\begin{proof}
It is already clear that $\theta_{\virmod_{\ell}}={\rm e}^{2\pi\bbmi \frweight{\ell}}\id_{\virmod_{\ell}}$.
When we apply the formula for \smash{$\theta_{\QGirrep{\ell}}^{-1}$} in~\eqref{eq:QGtwist_inverse} to the lowest weight vector \smash{$\QGirvec{\ell}{\ell}$}, we immediately get
\smash{$
 \theta_{\QGirrep{\ell}}^{-1} \QGirvec{\ell}{\ell}={\rm e}^{-2\pi\bbmi \frweight{\ell}}\QGirvec{\ell}{\ell}$}.
Therefore, we get~${\theta_{\QGirrep{\ell}}={\rm e}^{2\pi\bbmi \frweight{\ell}}\id_{\QGirrep{\ell}}}$ concluding the desired result.
\end{proof}

\section[The category of the C\_1-cofinite modules]{The category of the $\boldsymbol{ C_{1}}$-cofinite modules}
\label{sect:C1-cofinite_module_category}
We move on to looking at the category of $C_{1}$-cofinite modules of $\virvoa_{c}$.
In general, given a VOA $\voagen$ and $\voagen$-module $\voamodgen$,
we may form the $C_{1}$-space of $\voamodgen$ as
\begin{gather*}
	C_{1}(\voamodgen) = \bigl\{\bigl[x^{0}\bigr]\voaY_{\voamodgen}(\voavec,x)\vmodvec\mid \voavec\in\voagen_{>0}, \, \vmodvec \in \voamodgen\bigr\},
\end{gather*}
where the symbol $[x^{0}]$ takes the coefficient of $x^{0}$, i.e., the constant term, in the following formal series, and $\voaV_{>0} = \bigoplus_{n>0} \voaV_{n}$.
The module $\voamodgen$ is called $C_{1}$-cofinite, if $C_{1}(W)$ has a finite codimension: $\dim \voamodgen/C_{1}(\voamodgen)<\infty$.

We write $\catVir^{1}$ for the category of $C_{1}$-cofinite $\virvoa_{c}$-modules.
Due to~\cite{CJORY-tensor_categories_arising_from_Virasoro_algebra}, it is equipped with the structure of a braided tensor category,
resulting in $\catVir^{1}(t)$ depending on the parameter $t$.
In the same paper, \cite{CJORY-tensor_categories_arising_from_Virasoro_algebra} studied the detailed structure of $\catVir^{+}(t)$.
Combining their results and ours in the previous Section~\ref{sect:first-row_module_category},
Theorem~\ref{thm:C1-cat_QG} is rather straightforward as we will see below.

\subsection[Category catVir\^1(t)]{Category $\boldsymbol{ \catVir^{1}(t)}$}
Let us first see the main input from~\cite{CJORY-tensor_categories_arising_from_Virasoro_algebra} regarding the category $\catVir^{1}(t)$
of $C_{1}$-cofinite $\virvoa_{c}$-modules.

\begin{thm}[\cite{CJORY-tensor_categories_arising_from_Virasoro_algebra}]
\label{thm:c1cat-tensor}
The category $\catVir^{1}(t)$ is a semi-simple braided tensor category with simple objects $\virmod_{(k,\ell)}$, $k,\ell\in \bZ_{\geq 0}$.
Furthermore, we have the following fusion rules:
\begin{align*}
	\virmod_{(k_{1},0)}\boxtimes\virmod_{(k_{2},0)} = \bigoplus_{k_{3}\in \Sel (k_{1},k_{2})}\virmod_{(k_{3},0)},\qquad
	\virmod_{(k,0)}\boxtimes\virmod_{(0,\ell)} = \virmod_{(k,\ell)}
\end{align*}
for $k_{1},k_{2},k,\ell\in\bZ_{\geq 0}$.
\end{thm}

In particular, the first-row category $\catVir^{+}(t)$ can be found as a tensor subcategory of $\catVir^{1}(t)$.
Analogously to the first-row category, we may define the first-column category $\catVir^{-}(t)$
generated by the modules $\virmod_{(0,s)}$, $s\in \bZ_{\geq 0}$ in the first column of the Kac table.
Recall that the row and column are exchanged under changing the parameter $t$ to $t^{-1}$.
Therefore, the first-column category $\catVir^{-}(t)$ is also a braided tensor category and is equivalent to
$\catQG (\tilde{q})$ with another ${\tilde{q} = {\rm e}^{\pi \bbmi t^{-1}}}$.

We may define the (Deligne) tensor product $\catVir^{+}(t) \TCprod\catVir^{-} (t)$ following~\cite[Section~4.6]{etingof2015tensor}.
It is an additive category whose objects are direct sums of those of the form $X^{+}\TCprod X^{-}$ with $X^{\pm}\in \catVir^{\pm}(t)$.
The space of morphisms from $X^{+}\TCprod X^{-}$ to $Y^{+}\TCprod Y^{-}$ is given by
\begin{gather*}
	\Hom_{\catVir^{+}(t)\TCprod\catVir^{-}(t)}\big(X^{+}\TCprod X^{-}, Y^{+}\TCprod Y^{-}\big) = \Hom_{\catVir^{+}(t)}\big(X^{+},Y^{+}\big)\otimes_{\bC} \Hom_{\catVir^{-}(t)}(X^{-},Y^{-}).
\end{gather*}
From this definition, it follows that the simple objects of \smash{$\catVir^{+}(t)\TCprod\catVir^{-}(t)$} are \smash{$\virmod_{(k,0)}\TCprod \virmod_{(0,\ell)}$} with~${k,\ell \in \bZ_{\geq 0}}$.
The monoidal structures on $\catVir^{\pm}(t)$ are naturally transferred to $\catVir^{+}(t)\TCprod\catVir^{-}(t)$.
In particular, the associativity on $\catVir^{+}(t)\TCprod\catVir^{-}(t)$ is simply the product of those on $\catVir^{\pm}(t)$.

\begin{prop}
\label{prop:C1_category_decomposition}
The functor determined by
\begin{gather*}
	\catVir^{+}(t)\TCprod\catVir^{-}(t) \to \catVir^{1}(t);\qquad X^{+}\TCprod X^{-} \mapsto X^{+}\boxtimes X^{-}
\end{gather*}
and the natural isomorphisms
\begin{gather}
\label{eq:Jisom_C1}
	\id_{X^{+}}\boxtimes \braid_{X^{-},Y^{+}}\boxtimes \id_{Y^{-}}\colon\ X^{+}\boxtimes X^{-} \boxtimes Y^{+}\boxtimes Y^{-} \to X^{+}\boxtimes Y^{+} \boxtimes X^{-}\boxtimes Y^{-},
\end{gather}
with $X^{+},Y^{+}\in \catVir^{+}(t)$, $X^{-},Y^{-}\in \catVir^{-}(t)$ form an equivalence of tensor categories.
Here, $\braid$ is the braiding in $\catVir^{1}(t)$.
\end{prop}
\begin{proof}
To see that the functor is an equivalence of additive categories, it suffices to show that, for $X^{+}, Y^{+}\in \catVir^{+}(t)$ and $X^{-}, Y^{-}\in \catVir^{-}(t)$,
\begin{gather*}
	\Hom_{\catVir^{+}(t)}\big(X^{+},Y^{+}\big)\otimes_{\bC} \Hom_{\catVir^{-}(t)}(X^{-},Y^{-}) \to \Hom_{\catVir^{1}(t)}\big(X^{+}\boxtimes X^{-}, Y^{+}\boxtimes Y^{-}\big), \\
	f\otimes g \mapsto f\boxtimes g
\end{gather*}
is a linear isomorphism.
It is clearly the case as we can check on simple objects.

It remains to show that the functor along with the isomorphisms \eqref{eq:Jisom_C1} is a tensor functor.
In other words, we need to show that the isomorphisms
\begin{align*}
	\big(\big(X^{+}\boxtimes Y^{+}\big)\boxtimes Z^{+}\big) \boxtimes ((X^{-}\boxtimes Y^{-})\boxtimes Z^{-})
	\to \big(X^{+}\boxtimes \big(Y^{+}\boxtimes Z^{+}\big)\big) \boxtimes (X^{-}\boxtimes (Y^{-}\boxtimes Z^{-}))
\end{align*}
and
\begin{align*}
	\big(\big(X^{+}\boxtimes X^{-}\big)\boxtimes \big(Y^{+}\boxtimes Y^{-})\big)\boxtimes \big(Z^{+}\boxtimes Z^{-}\big)
	\to \big(X^{+}\boxtimes X^{-}\big)\boxtimes \big(\big(Y^{+}\boxtimes Y^{-}\big)\boxtimes \big(Z^{+}\boxtimes Z^{-}\big)\big)
\end{align*}
coincide for any choice of $X^{+},Y^{+},Z^{+}\in \catVir^{+}(t)$ and $X^{-},Y^{-},Z^{-}\in \catVir^{-}(t)$
under the identification by \eqref{eq:Jisom_C1}.
This follows from the identity
\begin{align*}
	\braid_{X^{-}\boxtimes Y^{-},Z^{+}}\circ \braid_{X^{-},Y^{+}} = \braid_{X^{-},Y^{+}\boxtimes Z^{+}}\circ \braid_{Y^{-},Z^{+}}
\end{align*}
and the axioms of a braided tensor category.
\end{proof}

\subsection[Braiding in catVir\^1(t)]{Braiding in $\boldsymbol{ \catVir^{1}(t)}$}
We already know the explicit formulas for the braiding in $\catVir^{\pm}(t)$.
Furthermore, in the whole category $\catVir^{1}(t)$, we can trivialize the braiding between $\catVir^{+}(t)$ and $\catVir^{-}(t)$ as follows.

For each $k,\ell\in\bZ_{\geq 0}$, let us fix an intertwining operator
\smash{$
	\voaint(-,x) \in \intsp{\virmod_{(k,0)}}{\virmod_{(0,\ell)}}{\virmod_{(k,\ell)}}$}.
By Theorem~\ref{thm:c1cat-tensor}, we may realize $\virmod_{(k,0)}\boxtimes\virmod_{(0,\ell)}$ as
$
	\virmod_{(k,0)}\boxtimes\virmod_{(0,\ell)} = \virmod_{(k,\ell)}
$ along with $\voaint(-,1)$
and $\virmod_{(0,\ell)}\boxtimes\virmod_{(k,0)}$ as
$
	\virmod_{(0,\ell)}\boxtimes\virmod_{(k,0)} = \virmod_{(k,\ell)}
$ along with $(\Omega^{-1}\voaint)(-,1)$.
Here, recall \eqref{eq:int_flip} for the transformation $\Omega$.

\begin{prop}
\label{prop:braiding_row_column}
For $k,\ell\in \bZ_{\geq 0}$, we assume the above realization of $\virmod_{(k,0)}\boxtimes\virmod_{(0,\ell)}$ and $\virmod_{(0,\ell)}\boxtimes\virmod_{(k,0)}$.
Then, the braiding $c_{\virmod_{(k,0)},\virmod_{(0,\ell)}}$ and $c_{\virmod_{(0,\ell)},\virmod_{(k,0)}}$ are both $\id_{\virmod_{(k,\ell)}}$.
\end{prop}
\begin{proof}
The identity \smash{$c_{\virmod_{(k,0)},\virmod_{(0,\ell)}} = \id_{\virmod_{(k,\ell)}}$} follows from the definition of the braiding.
Regarding the other one $c_{\virmod_{(0,\ell)},\virmod_{(k,0)}} = \id_{\virmod_{(k,\ell)}}$, it suffices to show that $\Omega \voaint = \Omega^{-1}\voaint$
for the above fixed intertwining operator $\voaint$.
This follows from the identity
\smash{$
	{\rm e}^{2\pi\bbmi(h_{k+1,\ell+1} - h_{k+1,0} - h_{0,\ell + 1})} = 1
$}
with the conformal weights in the Kac table.
\end{proof}

Note that the last identity ensures that $\catVir^{1}(t)$ has a ribbon structure
with the standard twist operator from Section~\ref{sect:HL_braiding}.

\subsection{Quantum group dual}
Recall that the parameter $t$ for the Virasoro algebra and the quantum parameter $q$ for $\Uq$ were related as $q = {\rm e}^{\pi \bbmi t}$.
Let us form the tensor product
$
	\Uq\otimes \Uqtilde
$
of Hopf algebras with \smash{$\tilde{q} = {\rm e}^{\pi \bbmi t^{-1}}$}.
We set \smash{$\catQGdouble (q,\tilde{q})$} to be the category of finite-dimensional representations of $\Uq\otimes \Uqtilde$ such that the eigenvalues of $K\otimes 1$ and $1 \otimes K$ are of the form $q^{n}$ and $\tilde{q}^{n}$ with $n\in \bZ$, respectively.
It is reasonable to call such modules of type I.
Now, it is clear that $\catQGdouble (q,\tilde{q})$ is equivalent to the tensor product $\catQG (q) \TCprod \catQG (\tilde{q})$.
Thus, composing Theorem~\ref{thm:equivalence_vir_QG} with Proposition~\ref{prop:C1_category_decomposition}, we get an equivalence
\begin{gather*}
	\catQGdouble (q,\tilde{q}) \to \catQG (q) \TCprod \catQG (\tilde{q}) \to \catVir^{+}(t)\TCprod \catVir^{-}(t) \to \catVir^{1}(t)
\end{gather*}
of tensor categories.

Furthermore, we can equip \smash{$\catQGdouble (q,\tilde{q})$} with a braiding by the tensor product of the universal $R$-matrices on $\Uq$ and $\Uqtilde$.
In other words, the two components $\catQG (q)$ and $\catQG (\tilde{q})$ are braided trivially in \smash{$\catQGdouble (q,\tilde{q})$}.
The ribbon structure of \smash{$\catQGdouble (q,\tilde{q})$} is also given by the tensor product of those on the two components.
Due to Proposition~\ref{prop:braiding_row_column}, the above tensor equivalence is a ribbon equivalence.

\subsection*{Acknowledgements}
The author is grateful to Kalle Kyt{\"o}l{\"a}, Eveliina Peltola, and Ingo Runkel for fruitful discussions.
The author also thanks the anonymous referees for various suggestions for improvement.
This work was supported by Academy of Finland (No.~248~130).

\pdfbookmark[1]{References}{ref}
\LastPageEnding


\begin{thebibliography}{99}
\footnotesize\itemsep=0pt

\bibitem{BakalovKirillovJr2001}
Bakalov B., Kirillov Jr. A., Lectures on tensor categories and modular
 functors, \textit{Univ. Lecture Ser.}, Vol.~21,
 \href{https://doi.org/10.1090/ulect/021}{American Mathematical Society},
 Providence, RI, 2001.

\bibitem{BB-CFTs_of_SLEs}
Bauer M., Bernard D., Conformal field theories of stochastic {L}oewner
 evolutions, \href{https://doi.org/10.1007/s00220-003-0881-x}{\textit{Comm.
 Math. Phys.}} \textbf{239} (2003), 493--521,
 \href{http://arxiv.org/abs/hep-th/0210015}{arXiv:hep-th/0210015}.

\bibitem{BPZ1984b}
Belavin A.A., Polyakov A.M., Zamolodchikov A.B., Infinite conformal symmetry in
 two-dimensional quantum field theory,
 \href{https://doi.org/10.1016/0550-3213(84)90052-X}{\textit{Nuclear Phys.~B}}
 \textbf{241} (1984), 333--380.

\bibitem{Cardy_lecture_2008}
Cardy J., Conformal field theory and statistical mechanics, in Exact {M}ethods
 in {L}ow-{D}imensional {S}tatistical {P}hysics and {Q}uantum {C}omputing,
 Oxford University Press, Oxford, 2010, 65--98,
 \href{http://arxiv.org/abs/0807.3472}{arXiv:0807.3472}.

\bibitem{chari1995guide}
Chari V., Pressley A., A guide to quantum groups, Cambridge University Press,
 Cambridge, 1995.

\bibitem{CJORY-tensor_categories_arising_from_Virasoro_algebra}
Creutzig T., Jiang C., Orosz~Hunziker F., Ridout D., Yang J., Tensor categories
 arising from the {V}irasoro algebra,
 \href{https://doi.org/10.1016/j.aim.2021.107601}{\textit{Adv. Math.}}
 \textbf{380} (2021), 107601, 35~pages,
 \href{http://arxiv.org/abs/2002.03180}{arXiv:2002.03180}.

\bibitem{creutzig2023algebraic}
Creutzig T., Lentner S., Rupert M., An algebraic theory for logarithmic
 {K}azhdan--{L}usztig correspondences,
 \href{http://arxiv.org/abs/2306.11492}{arXiv:2306.11492}.

\bibitem{CreutzigLentnerRupert}
Creutzig T., Lentner S., Rupert M., Characterizing braided tensor categories
 associated to logarithmic vertex operator algebras,
 \href{http://arxiv.org/abs/2104.13262}{arXiv:2104.13262}.

\bibitem{DMS_yellow_CFT}
Di~Francesco P., Mathieu P., S\'en\'echal D., Conformal field theory, \textit{Grad. Texts Contemp. Phys.},
 \href{https://doi.org/10.1007/978-1-4612-2256-9}{Springer}, New York, 1997.

\bibitem{Drinfeld1986}
Drinfeld V.G., Quantum groups, in Proceedings of the {I}nternational {C}ongress
 of {M}athematicians, {V}ol.~1,~2 ({B}erkeley, {C}alif., 1986), American
 Mathematical Society, Providence, RI, 1987, 798--820.

\bibitem{Drinfeld1989}
Drinfeld V.G., Quasi-{H}opf algebras, \textit{Leningrad Math.~J.} \textbf{1}
 (1990), 1419--1457.

\bibitem{etingof2015tensor}
Etingof P., Gelaki S., Nikshych D., Ostrik V., Tensor categories, \textit{Math.
 Surveys Monogr.}, Vol. 205, \href{https://doi.org/10.1090/surv/205}{American
 Mathematical Society}, Providence, RI, 2015.

\bibitem{FeiginGainutdinovSemikhatovTipunin2006a}
Feigin B.L., Gainutdinov A.M., Semikhatov A.M., Tipunin I.Yu., Modular group
 representations and fusion in logarithmic conformal field theories and in the
 quantum group center,
 \href{https://doi.org/10.1007/s00220-006-1551-6}{\textit{Comm. Math. Phys.}}
 \textbf{265} (2006), 47--93,
 \href{http://arxiv.org/abs/hep-th/0504093}{arXiv:hep-th/0504093}.

\bibitem{FW-topological_representation_of_Uqsl2}
Felder G., Wieczerkowski C., Topological representations of the quantum
 group~{$U_q({\rm sl}_2)$},
 \href{https://doi.org/10.1007/BF02102043}{\textit{Comm. Math. Phys.}}
 \textbf{138} (1991), 583--605.

\bibitem{Finkelberg1996}
Finkelberg M., An equivalence of fusion categories,
 \href{https://doi.org/10.1007/BF02247887}{\textit{Geom. Funct. Anal.}}
 \textbf{6} (1996), 249--267, {E}rratum,
 \href{https://doi.org/10.1007/s00039-013-0230-y}{\textit{Geom. Funct. Anal.}}
 \textbf{23} (2013), 249--267.

\bibitem{FrenkelBen-Zvi2004}
Frenkel E., Ben-Zvi D., Vertex algebras and algebraic curves, 2nd ed., \textit{Math.
 Surveys Monogr.}, Vol.~88,
 \href{https://doi.org/10.1090/surv/088}{American Mathematical Society},
 Providence, RI, 2004,
 \href{http://arxiv.org/abs/math.QA/0007054}{arXiv:math.QA/0007054}.

\bibitem{frenkel1997canonical}
Frenkel I., Khovanov M.G., Canonical bases in tensor products and graphical
 calculus for~{$U_q({\mathfrak s}{\mathfrak l}_2)$},
 \href{https://doi.org/10.1215/S0012-7094-97-08715-9}{\textit{Duke Math.~J.}}
 \textbf{87} (1997), 409--480.

\bibitem{FLM-VOAs_Monster}
Frenkel I., Lepowsky J., Meurman A., Vertex operator algebras and the
 {M}onster, \textit{Pure Appl. Math.}, Vol.~134, Academic Press, Boston, MA,
 1988.

\bibitem{FrenkelZhu2012}
Frenkel I., Zhu M., Vertex algebras associated to modified regular
 representations of the {V}irasoro algebra,
 \href{https://doi.org/10.1016/j.aim.2012.02.008}{\textit{Adv. Math.}}
 \textbf{229} (2012), 3468--3507,
 \href{http://arxiv.org/abs/1012.5443}{arXiv:1012.5443}.

\bibitem{FeiginGainutdinovSemikhatovTipunin2006b}
Gainutdinov A.M., Semikhatov A.M., Tipunin I.Yu., Feigin B.L., The
 {K}azhdan--{L}usztig correspondence for the representation category of the
 triplet {$W$}-algebra in logorithmic {CFT},
 \href{https://doi.org/10.1007/s11232-006-0113-6}{\textit{Theoret. and Math.
 Phys.}} \textbf{148} (2006), 1210--1235,
 \href{http://arxiv.org/abs/math.QA/0512621}{arXiv:math.QA/0512621}.

\bibitem{GannonNegron-quantum_sltwo_and_log_VOAs}
Gannon T., Negron C., Quantum~$\mathrm{SL}(2)$ and logarithmic vertex operator
 algebras at $(p,1)$-central charge,
 \href{https://doi.org/10.4171/JEMS/1489}{\textit{J.~Eur. Math. Soc. (JEMS)}},
{t}o appear,
 \href{http://arxiv.org/abs/2104.12821}{arXiv:2104.12821}.

\bibitem{GRS-quantum_groups_in_2d_physics}
G\'omez C.G., Ruiz-Altaba M., Sierra G., Quantum groups in two-dimensional
 physics, \textit{Cambridge Monogr. Math. Phys.},
 \href{https://doi.org/10.1017/CBO9780511628825}{Cambridge University Press}, Cambridge,
 1996.

\bibitem{GreenSchwarzWitten_superstring}
Green M.B., Schwarz J.H., Witten E., Superstring theory~I,~II, \textit{Cambridge
 Monogr. Math. Phys.}, \href{https://doi.org/10.1002/asna.2113090428}{Cambridge
 University Press}, Cambridge, 1987.

\bibitem{Hansson_QHall_CFT}
Hansson T.H., Hermanns M., Simon S.H., Viefers S.F., Quantum hall physics: {H}ierarchies and conformal field theory techniques,
 \href{https://doi.org/10.1103/RevModPhys.89.025005}{\textit{Rev. Modern
 Phys.}} \textbf{89} (2017), 025005, 61~pages,
 \href{http://arxiv.org/abs/1601.01697}{arXiv:1601.01697}.

\bibitem{Huang1995}
Huang Y.-Z., A theory of tensor products for module categories for a vertex
 operator algebra.~{IV},
 \href{https://doi.org/10.1016/0022-4049(95)00050-7}{\textit{J.~Pure Appl.
 Algebra}} \textbf{100} (1995), 173--216,
 \href{http://arxiv.org/abs/q-alg/9505019}{arXiv:q-alg/9505019}.

\bibitem{Huang-CFT_and_VOA}
Huang Y.-Z., Two-dimensional conformal geometry and vertex operator algebras,
 \textit{Progr. Math.}, Vol.~148,
 \href{https://doi.org/10.1007/978-1-4612-4276-5}{Birkh\"auser, Boston, MA},
 1997.

\bibitem{Huang2005}
Huang Y.-Z., Differential equations and intertwining operators,
 \href{https://doi.org/10.1142/S0219199705001799}{\textit{Commun. Contemp.
 Math.}} \textbf{7} (2005), 375--400,
 \href{http://arxiv.org/abs/math.QA/0206206}{arXiv:math.QA/0206206}.

\bibitem{Huang--Kirillov--Lepowsky}
Huang Y.-Z., Kirillov Jr. A., Lepowsky J., Braided tensor categories and
 extensions of vertex operator algebras,
 \href{https://doi.org/10.1007/s00220-015-2292-1}{\textit{Comm. Math. Phys.}}
 \textbf{337} (2015), 1143--1159,
 \href{http://arxiv.org/abs/1406.3420}{arXiv:1406.3420}.

\bibitem{HuangLepowsky1992}
Huang Y.-Z., Lepowsky J., Toward a theory of tensor products for representations
 of a vertex operator algebra, in Proceedings of the {XX}th {I}nternational
 {C}onference on {D}ifferential {G}eometric {M}ethods in {T}heoretical
 {P}hysics, {V}ol.~1,~2 ({N}ew {Y}ork, 1991), World Scientific Publishing,
 River Edge, NJ, 1992, 344--354.

\bibitem{HuangLepowsky1994}
Huang Y.-Z., Lepowsky J., Tensor products of modules for a vertex operator
 algebra and vertex tensor categories, in Lie {T}heory and {G}eometry,
 \textit{Progr. Math.}, Vol.~123,
 \href{https://doi.org/10.1007/978-1-4612-0261-5_13}{Birkh\"auser}, Boston,
 MA, 1994, 349--383.

\bibitem{Huang--LepowskyI}
Huang Y.-Z., Lepowsky J., A theory of tensor products for module categories for
 a vertex operator algebra.~{I},
 \href{https://doi.org/10.1007/BF01587908}{\textit{Selecta Math. (N.S.)}}
 \textbf{1} (1995), 699--756,
 \href{http://arxiv.org/abs/hep-th/9309076}{arXiv:hep-th/9309076}.

\bibitem{Huang--LepowskyII}
Huang Y.-Z., Lepowsky J., A theory of tensor products for module categories for
 a vertex operator algebra.~{II},
 \href{https://doi.org/10.1007/BF01587908}{\textit{Selecta Math. (N.S.)}}
 \textbf{1} (1995), 757--786,
 \href{http://arxiv.org/abs/hep-th/9309076}{arXiv:hep-th/9309076}.

\bibitem{Huang--LepowskyIII}
Huang Y.-Z., Lepowsky J., A theory of tensor products for module categories for
 a vertex operator algebra.~{III},
 \href{https://doi.org/10.1016/0022-4049(95)00049-3}{\textit{J.~Pure Appl.
 Algebra}} \textbf{100} (1995), 141--171,
 \href{http://arxiv.org/abs/q-alg/9505018}{arXiv:q-alg/9505018}.

\bibitem{huang1999intertwining}
Huang Y.-Z., Lepowsky J., Intertwining operator algebras and vertex tensor
 categories for affine {L}ie algebras,
 \href{https://doi.org/10.1215/S0012-7094-99-09905-2}{\textit{Duke Math.~J.}}
 \textbf{99} (1999), 113--134,
 \href{http://arxiv.org/abs/q-alg/9706028}{arXiv:q-alg/9706028}.

\bibitem{IK-representation_theory_of_the_Virasoro_algebra}
Iohara K., Koga Y., Representation theory of the {V}irasoro algebra, \textit{Springer
 Monogr. Math.}, \href{https://doi.org/10.1007/978-0-85729-160-8}{Springer}, London, 2011.

\bibitem{Kac-vertex_algebras}
Kac V., Vertex algebras for beginners, 2nd~ed.,  \textit{Univ. Lecture Ser.}, Vol.~10, \href{https://doi.org/10.1090/ulect/010}{American Mathematical
 Society}, Providence, RI, 1998.

\bibitem{Kassel-quantum_groups}
Kassel C., Quantum groups, \textit{Grad. Texts in Math.}, Vol.~155,
 \href{https://doi.org/10.1007/978-1-4612-0783-2}{Springer, New York}, 1995.

\bibitem{KL-tensor_structures_affine_Lie-I}
Kazhdan D., Lusztig G., Tensor structures arising from affine {L}ie
 algebras.~{I}, \href{https://doi.org/10.2307/2152745}{\textit{J. Amer. Math.
 Soc.}} \textbf{6} (1993), 905--947.

\bibitem{KL-tensor_structures_affine_Lie-II}
Kazhdan D., Lusztig G., Tensor structures arising from affine {L}ie
 algebras.~{II}, \href{https://doi.org/10.2307/2152746}{\textit{J. Amer. Math.
 Soc.}} \textbf{6} (1993), 949--1011.

\bibitem{KL-tensor_structures_affine_Lie-III}
Kazhdan D., Lusztig G., Tensor structures arising from affine {L}ie algebras.
 {III}, \href{https://doi.org/10.2307/2152762}{\textit{J. Amer. Math. Soc.}}
 \textbf{7} (1994), 335--381.

\bibitem{KL-tensor_structures_affine_Lie-IV}
Kazhdan D., Lusztig G., Tensor structures arising from affine {L}ie
 algebras.~{IV}, \href{https://doi.org/10.2307/2152763}{\textit{J.~Amer. Math.
 Soc.}} \textbf{7} (1994), 383--453.

\bibitem{KondoSaito2011}
Kondo H., Saito Y., Indecomposable decomposition of tensor products of modules
 over the restricted quantum universal enveloping algebra associated
 to~{${\mathfrak{sl}}_2$},
 \href{https://doi.org/10.1016/j.jalgebra.2011.01.010}{\textit{J.~Algebra}}
 \textbf{330} (2011), 103--129,
 \href{http://arxiv.org/abs/0901.4221}{arXiv:0901.4221}.

\bibitem{KoshidaKytola2022}
Koshida S., Kyt\"ol\"a K., The quantum group dual of the first-row subcategory
 for the generic {V}irasoro {VOA},
 \href{https://doi.org/10.1007/s00220-021-04266-w}{\textit{Comm. Math. Phys.}}
 \textbf{389} (2022), 1135--1213,
 \href{http://arxiv.org/abs/2105.13839}{arXiv:2105.13839}.

\bibitem{KP-conformally_covariant_bdry_correlations}
Kyt\"ol\"a K., Peltola E., Conformally covariant boundary correlation functions
 with a quantum group, \href{https://doi.org/10.4171/jems/917}{\textit{J.~Eur.
 Math. Soc. (JEMS)}} \textbf{22} (2020), 55--118,
 \href{http://arxiv.org/abs/1408.1384}{arXiv:1408.1384}.

\bibitem{lentner2025conditional}
Lentner S.D., A conditional algebraic proof of the logarithmic
 {K}azhdan--{L}usztig correspondence,
 \href{http://arxiv.org/abs/2501.10735}{arXiv:2501.10735}.

\bibitem{Lepowsky_Li-VOA}
Lepowsky J., Li H., Introduction to vertex operator algebras and their
 representations, \textit{Progr. Math.}, Vol.~227,
 \href{https://doi.org/10.1007/978-0-8176-8186-9}{Birkh\"auser}, Boston, MA,
 2004.

\bibitem{Ludwig_CFT_condensed_matter}
Ludwig W.W., Methods of conformal field theory in condensed matter physics, in
 Low-{D}imensional {Q}uantum {F}ield {T}heories for {C}ondensed {M}atter
 {P}hysicists, \textit{Ser. Modern Condensed Matter Phys.}, Vol.~6,
 \href{https://doi.org/10.1142/9789814447027_0007}{World Scientific
 Publishing}, 1995, 389--455.

\bibitem{Lusztig1993}
Lusztig G., Introduction to quantum groups, \textit{Progr. Math.}, Vol.~110,
 \href{https://doi.org/10.1007/978-3-031-23817-8}{Birkh\"auser}, Boston, MA,
 1993.

\bibitem{masbaum19943}
Masbaum G., Vogel P., {$3$}-valent graphs and the {K}auffman bracket,
 \href{https://doi.org/10.2140/pjm.1994.164.361}{\textit{Pacific~J. Math.}}
 \textbf{164} (1994), 361--381.

\bibitem{McRae-nonneg_integer_level_affine_Lie_algebra_tensor_cat}
McRae R., Non-negative integral level affine {L}ie algebra tensor categories
 and their associativity isomorphisms,
 \href{https://doi.org/10.1007/s00220-016-2683-y}{\textit{Comm. Math. Phys.}}
 \textbf{346} (2016), 349--395,
 \href{http://arxiv.org/abs/1506.00113}{arXiv:1506.00113}.

\bibitem{McRaeYang2022}
McRae R., Yang J., An {$\mathfrak{sl}_2$}-type tensor category for the
 {V}irasoro algebra at central charge~25 and applications,
 \href{https://doi.org/10.1007/s00209-022-03197-z}{\textit{Math.~Z.}}
 \textbf{303} (2023), 32, 40~pages,
 \href{http://arxiv.org/abs/2202.07351}{arXiv:2202.07351}.

\bibitem{MR-comment_on_quantum_group_symmetry_in_CFT}
Moore G., Reshetikhin N., A comment on quantum group symmetry in conformal
 field theory,
 \href{https://doi.org/10.1016/0550-3213(89)90219-8}{\textit{Nuclear Phys.~B}}
 \textbf{328} (1989), 557--574.

\bibitem{Mussardo-statistical_field_theory}
Mussardo G., Statistical field theory, \textit{Oxf. Grad. Texts}, Oxford University
 Press, Oxford, 2010.

\bibitem{NagatomoTsuchiya2011}
Nagatomo K., Tsuchiya A., The triplet vertex operator algebra~{$W(p)$} and the
 restricted quantum group {$\overline U_q(sl_2)$} at~{$q={\rm e}^{\frac{\pi
 i}{p}}$}, in Exploring {N}ew {S}tructures and {N}atural {C}onstructions in
 {M}athematical {P}hysics, \textit{Adv. Stud. Pure Math.}, Vol.~61,
 \href{https://doi.org/10.2969/aspm/06110001}{Math. Soc. Japan}, Tokyo, 2011,
 1--49, \href{http://arxiv.org/abs/0902.4607}{arXiv:0902.4607}.

\bibitem{nakano2024fusion}
Nakano H., Hunziker F.O., Camacho A.R., Wood S., Fusion rules and rigidity for
 weight modules over the simple admissible affine~$\mathfrak{sl}_{2}$
 and~${\mathcal{N}=2}$ superconformal vertex operator superalgebras,
 \href{http://arxiv.org/abs/2411.11387}{arXiv:2411.11387}.

\bibitem{ostrik2005module}
Ostrik V., Module categories over representations of~${{\rm SL}_q (2)}$ in the
 non-semisimple case,
 \href{http://arxiv.org/abs/math.QA/0509530}{arXiv:math.QA/0509530}.

\bibitem{PS-common_structures_between_finite_systems_and_CFTs}
Pasquier V., Saleur H., Common structures between finite systems and conformal
 field theories through quantum groups,
 \href{https://doi.org/10.1016/0550-3213(90)90122-T}{\textit{Nuclear Phys.~B}}
 \textbf{330} (1990), 523--556.

\bibitem{RRR-contour_picture_of_quantum_groups_in_CFT}
Ram\'{\i}rez C., Ruegg H., Ruiz-Altaba M., The contour picture of quantum
 groups: conformal field theories,
 \href{https://doi.org/10.1016/0550-3213(91)90583-J}{\textit{Nuclear Phys.~B}}
 \textbf{364} (1991), 195--233.

\bibitem{SV-quantum_groups_and_homology_of_local_systems}
Schechtman V.V., Varchenko A.N., Quantum groups and homology of local systems,
 in Algebraic {G}eometry and {A}nalytic {G}eometry ({T}okyo, 1990), ICM-90
 Satell. Conf. Proc.,
 \href{https://doi.org/10.1007/978-4-431-68172-4_10}{Springer}, Tokyo, 1991,
 182--197.

\bibitem{Schramm2000}
Schramm O., Scaling limits of loop-erased random walks and uniform spanning
 trees, \href{https://doi.org/10.1007/BF02803524}{\textit{Israel~J. Math.}}
 \textbf{118} (2000), 221--288,
 \href{http://arxiv.org/abs/math.PR/9904022}{arXiv:math.PR/9904022}.

\bibitem{TsuchiyaWood2013}
Tsuchiya A., Wood S., The tensor structure on the representation category of
 the {$\mathcal{W}_p$} triplet algebra,
 \href{https://doi.org/10.1088/1751-8113/46/44/445203}{\textit{J.~Phys.~A}}
 \textbf{46} (2013), 445203, 40~pages,
 \href{http://arxiv.org/abs/1201.0419}{arXiv:1201.0419}.

\bibitem{Varchenko-multidimensional_hypergeometric_functions_and_representation_theory_of_Lie_algebras_and_quantum_groups}
Varchenko A., Multidimensional hypergeometric functions and representation
 theory of {L}ie algebras and quantum groups, \textit{Adv. Ser. Math. Phys.},
 Vol.~21, \href{https://doi.org/10.1142/2467}{World Scientific Publishing Co.},
 River Edge, NJ, 1995.

\end{thebibliography}
\end{document}